\documentclass[11pt]{article}
\usepackage{amsmath,amsfonts,amsthm}
\usepackage[margin=1in]{geometry}
\usepackage[normalem]{ulem}
\usepackage{amssymb}
\usepackage{appendix}

\usepackage[hyphens]{url}
\usepackage{hyperref}
\usepackage{graphicx}
\usepackage{epstopdf}
\usepackage{float}
\usepackage{subfigure}
\usepackage{verbatim}
\usepackage{enumitem}
\usepackage{color}
\usepackage{bm}
\usepackage{dsfont}

\usepackage{soul}

\numberwithin{equation}{section}

\theoremstyle{plain}

\newcommand{\lam}{\lambda}
\newcommand{\tU}{\tilde{U}}
\newcommand{\tu}{\tilde{u}}

\newcommand{\tv}{\tilde{V}}

\newtheorem{lemma}{Lemma}[section]
\newtheorem{theorem}{Theorem}[section]

\newtheorem{proposition}{Proposition}[section]
\newtheorem{corollary}{Corollary}[section]

\newtheorem{remark}{Remark}[section]
\allowdisplaybreaks[1]

\def \E{\mathbb{E}}

\def\cred{\color{blue}}
\def\cblue{\color{blue}}

\makeatletter
\def\@setcopyright{}
\def\serieslogo@{}
\makeatother

\title{Optimal Consumption with Loss Aversion and Reference to Past Spending Maximum}
\author{Xun LI\thanks{Department of Applied Mathematics, The Hong Kong Polytechnic University, Hung Hom, Kowloon, Hong Kong. Email:\texttt{li.xun@polyu.edu.hk}}
\and
Xiang YU\thanks{Department of Applied Mathematics, The Hong Kong Polytechnic University, Hung Hom, Kowloon, Hong Kong. Email:\texttt{xiang.yu@polyu.edu.hk}}
\and
Qinyi ZHANG\thanks{Department of Applied Mathematics, The Hong Kong Polytechnic University, Hung Hom, Kowloon, Hong Kong. Email:\texttt{qinyi-qy.zhang@connect.polyu.hk}}
}

\begin{document}

\date{\vspace{-0.7cm}}
\maketitle

\begin{abstract}
This paper studies an optimal consumption problem for a loss-averse agent with reference to past consumption maximum. To account for loss aversion on relative consumption, an S-shaped utility is adopted that measures the difference between the non-negative consumption rate and a fraction of the historical spending peak. We consider the concave envelope of the utility with respect to consumption, allowing us to focus on an auxiliary HJB variational inequality on the strength of concavification principle and dynamic programming arguments. By applying the dual transform and smooth-fit conditions, the auxiliary HJB variational inequality is solved in piecewise closed-form and some thresholds of the wealth variable are obtained. The optimal consumption and investment control can be derived in the piecewise feedback form. The rigorous verification proofs on optimality and concavification principle are provided. Some numerical sensitivity analysis and financial implications are also presented.
\end{abstract}

\section{Introduction}\label{sec: intro}
Optimal portfolio-consumption via utility maximization has been one of the fundamental research topics in mathematical finance.
In the seminal works of Merton \cite{Mert1969RES, Mert1971JET}, the feedback optimal investment and consumption strategy is first derived by resorting to dynamic programming arguments and the solution of the associated HJB equation.
Since then, abundant influential results and methodology have been rapidly developed to accommodate more general financial market models, trading constraints and other factors in decision making.
Giving a complete list of references is beyond the scope of this paper. To partially explain the smooth consumption behavior, it has been suggested in the literature to take into account the past consumption decision in the measurement of the utility function.
By considering the relative consumption with respect to a reference that depends on the past consumption, the striking changes in consumption can essentially be ruled out from the optimal solution.
The widely used habit formation preference (see Abel \cite{Abel1990AER}, Constantinides \cite{Constantinides1990JPE}, Detemple and Zapatero \cite{DetempleZapatero1992MF}) recommends the utility maximization problem as
\begin{align*}
\sup_{(\pi,c)\in\mathcal{A}}\mathbb{E}\left[\int_0^{\infty}e^{-\rho t}U(c_t-Z_t)dt\right],
\end{align*}
where $(Z_t)_{t\geq 0}$ stands for the habit formation process taking the form $Z_t=ze^{-\alpha t}+\int_0^t \delta e^{\alpha(s-t)dt}c_sds$ with discount factors $\alpha, \delta\geq 0$ and the initial habit $z\geq 0$. That is, the satisfaction and risk aversion of the agent depend on the relative deviation of the current consumption from the weighted average of the past consumption integral. Along this direction, some recent developments can be found in Schroder and Skiadas \cite{SchroderSkiadas2002RFS}, Detemple and Karatzas \cite{DetempleKaratzas2003JET}, Englezos and Karatzas \cite{EnglezosKaratzas2009Sicon}, Yang and Yu \cite{Yu2022}, Yu \cite{Yu2015AoAP, Yu2017AoAP} and references therein.

One notable advantage of the habit formation preference is its linear dependence on consumption, which enables one to consider $\hat{c}_t=c_t-Z_t$ as an auxiliary control in a fictitious market model so that the path-dependence can be hidden. This insightful transform, first observed in \cite{SchroderSkiadas2002RFS}, reduces the complexity of the problem significantly. The martingale and duality approach can be applied by considering the adjusted martingale measure density process essentially based on Fubini theorem; see Detemple and Karatzas \cite{DetempleKaratzas2003JET} and Yu \cite{Yu2015AoAP, Yu2017AoAP}.

Another stream of research on the consumption reference focuses on its historical maximum level. Indeed, a large expenditure might signal the turning point of one's standard of living and is usually a decision after careful thought and consideration. Such historical high spending moments are consequent on adequate wealth accumulation and often give rise to some long term subsequent consumption decisions such as maintenance, repairs and upgrade.
To take into account the impact of the past consumption maximum, some previous studies incorporate the ratcheting or drawdown constraints that $c_t\geq \lambda H_t$ into the Merton optimal consumption problem where $H_t:=\max\{h, \sup_{s\leq t}c_s\}$ stands for the consumption running maximum process with the initial level $h$ and $\lam\in(0,1]$ captures the degree of adherence towards the reference level $H$. In particular, the case $\lambda=1$ accounts for the ratcheting consumption constraint as studied in Dybvig \cite{Dyb1995}, and the case $0<\lambda<1$ depicts the consumption drawdown constraint as discussed in Arun \cite{Arun2020arXiv} and Angoshtari et al. \cite{AngBay}. 

Meanwhile, it is also of great importance to understand the consumption behavior when the past spending maximum appears inside the utility. By taking the multiplicative form of reference, Guasoni et al. \cite{GuasoniHubermanR2020MFF} and Li et al. \cite{LYZ23} adopt the Cobb-Douglas utility with a zero discount factor and study the problem
\begin{align*}
\sup_{(\pi,c)\in\mathcal{A}, c\leq h}\mathbb{E}\left[\int_0^{\infty}(c_t/H_t^{\alpha})^p/p\ dt\right].
\end{align*}
Recently, Deng et al. \cite{DengLiPY2020arXiv} investigate an optimal consumption problem bearing the impact of the past spending maximum in the same form of the habit formation preference, which is defined by
\begin{align*}
\sup_{(\pi,c)\in\mathcal{A}, 0\leq c\leq h}\mathbb{E}\left[\int_0^{\infty}e^{-\rho t}U(c_t-\lambda H_t)dt\right].
\end{align*}
The exponential utility $U(x)=-e^{-\beta x}$ is considered therein and the non-negative consumption constraint $c_t\geq 0$ is enforced, yielding more regions for different consumption behavior. Although the running maximum term complicates the objective functional, the optimal consumption problems in both Guasoni et al. \cite{GuasoniHubermanR2020MFF} and Deng et al. \cite{DengLiPY2020arXiv} can be tackled successfully under the umbrella of dynamic programming.
The associated HJB variational inequalities and the feedback optimal controls can be solved in closed-form piecewisely in different regions, and some explicit and interpretable thresholds of the wealth are obtained.
One key feature in Guasoni et al. \cite{GuasoniHubermanR2020MFF} and Deng et al. \cite{DengLiPY2020arXiv} is to allow the agent to strategically consume below the reference level.
Nevertheless, from the behavioral finance perspective, one shortcoming in these studies is their incapability to distinguish agent's different risk aversion on the same-sized overperformance and falling behind with respect to the reference process. Instead, loss aversion depicts the agent's proclivity to prefer avoiding losses to acquiring equivalent gains, naturally leads to different left and right derivatives of the utility function at the reference point. In particular, it is an open problem how the loss aversion on consumption with respect to the past consumption peak may affect the optimal investment and consumption behavior.

The loss aversion with a reference point has been studied in behavioral finance predominantly on terminal wealth optimization, see among Berkelaar et al. \cite{BerklKouwbgPost2004RES}, Jin and Zhou \cite{JinZhou2008MF}, He and Zhou \cite{HeZhou2011MS, HeZhou2014QF}, He and Strub \cite{HeStrub2019Preprint}, He and Yang \cite{HeYang2019MF} and references therein. Only a handful of papers can be found to encode that the agent may hurt more when the consumption is falling below a reference, especially when the reference level is endogenously generated by past decisions. Recently, Curatola \cite{Curatola2017QREF} studies a utility maximization problem on consumption for a loss averse agent under an S-shaped utility when the reference is chosen as a specific integral of the past consumption process. Later, van Bilsen et al. \cite{BilsenLaevenN2020MS} consider a similar problem under a two-part utility when the reference process is defined as the conventional consumption habit formation process. By imposing some artificial lower bounds on consumption control, the martingale and duality approach together with the concavification principle can be employed in both papers.

By contrast, the present paper investigates the optimal consumption behavior of a loss-averse agent who feels differently when the consumption is over-performing and falling below
the past spending maximum. As the first attempt to combine the loss-aversion on relative consumption and the reference to historical consumption peak, the mathematical problem is formulated by
\begin{align*}
\sup_{(\pi,c)\in\mathcal{A}, ~0\leq c\leq h}\mathbb{E}\left[\int_0^{\infty}e^{-\rho t}U(c_t-\lambda H_t)dt\right],
\end{align*}
where $U(x)$ is described by the conventional two-part power utility (see Kahneman and Tversky \cite{KahnemanTversky2013Book}) that
\begin{align}\label{eq: utility}
U(x):=&\left\{\begin{aligned}
&\frac{x^{\beta_1}}{\beta_1},\quad   &  \mbox{if } x\geq 0, \\
&-k\frac{(-x)^{\beta_2}}{\beta_2},\quad & \mbox{if } x<0.
\end{aligned}\right.
\end{align}
Here, $k > 0$ stands for the loss aversion degree, and it is assumed in the present paper that $0<\beta_1,\beta_2<1$, which represent the risk aversion parameters over the gain domain $x\geq 0$ and the loss domain $x<0$, respectively. The utility is an S-shaped function on $\mathbb{R}$. The parameter $\lam\in(0,1]$ again reflects the degree of adherence towards the reference level $H$, which now affects the expected utility directly.

Our aim is to solve this stochastic control problem by dynamic programming arguments and the PDE approach. However, the non-concave utility introduces more challenges in solving the HJB variational inequality heuristically. We propose to focus on the realization utility $U(c-\lambda h)$ for each fixed $\lambda h$ and the control constraint $0\leq c\leq h$ and consider the concave envelope of $U(x-\lambda h)$ only with respect to the variable $0\leq x\leq h$ on the strength of concavification principle.
Similar to Deng et al. \cite{DengLiPY2020arXiv}, by considering both the wealth level $x$ and reference level $h$ as state variables, we can derive the auxiliary HJB variational inequality in the piecewise form based on the decomposition of the domain $(x,h)\in\mathbb{R}_+^2$ when the feedback optimal consumption: (i) equals 0; (ii) lies between 0 and past spending maximum; (iii) coincides with the past spending peak.
By utilizing the dual transformation only with respect to the wealth variable $x$ and treating the reference variable $h$ as a parameter, we arrive at a piecewise dual ODE problem in different regions. Together with some intrinsic boundary conditions and smooth-fit conditions, we are able to solve the piecewise ODE problem.

As opposed to the exponential utility in Deng et al. \cite{DengLiPY2020arXiv}, the concave envelope function in the current setting has no explicit form, which complicates the smooth fit arguments significantly. As a direct consequence, all coefficient functions in the solution to the dual ODE contain are implicit functions of the reference variable $h$.
After the inverse transform, all boundary curves separating different regions can still be expressed as thresholds of the wealth variable, albeit implicitly. The feedback optimal controls can also be derived analytically in terms of $x$ and $h$, in which the optimal consumption may exhibit jumps.
On account of the specific feedback form of optimal consumption in each region, the verification theorem on the optimality and concavification principle can be rigorously proved, giving the desired equivalence between the original problem and the auxiliary one using the concave envelope of the realization utility. When $U(x)$ in \eqref{eq: utility} is an S-shaped utility, it is interesting to observe that the optimal consumption $c_t^*$ exhibits a jump and it is either zero or above the reference $\lambda H^*_t$. That is, because the agent is risk-loving in the loss domain, she can never tolerate any positive consumption below the reference when the wealth is not sufficient and prefers to stop consumption right away to accumulate more capital from the financial market to sustain her future high consumption.

The rest of the paper is organized as follows.
Section \ref{sec: formulate} introduces the market model and the optimal consumption problem under the two-part utility with reference to past spending maximum. By considering the concave envelope of the utility function, we transform the original problem into an equivalent control problem.
In Section \ref{sec: main}, we solve the auxiliary HJB variational inequality.
The optimal controls for the original problem are obtained in piecewise feedback form across different regions and all boundary curves are derived analytically.
Section \ref{sec: long_run} presents some quantitative properties of the optimal controls and some numerical sensitivity analysis and their financial implications. In Section \ref{sec: proof}, we prove the verification theorem on optimality and concavification principle as well as some auxiliary results in previous sections.

\section{Model Setup and Problem Formulation}\label{sec: formulate}

\subsection{Market Model and Preference}
Let the filtered probability space $(\Omega, \mathcal{F}, \mathbb{F}, \mathbb{P})$ with $\mathbb{F}=(\mathcal{F}_t)_{t\geq 0}$ satisfy the usual conditions. The financial market consists of one riskless asset and one risky asset. The riskless asset price follows $dB_t=rB_tdt$,
where $r>0$ is the interest rate.
The risky asset price is governed by the SDE
\begin{equation}
dS_t=S_t\mu dt+S_t\sigma dW_t,\quad t\geq 0\nonumber
\end{equation}
where $W$ is a $\mathbb{F}$-adapted Brownian motion, and $\mu\in\mathbb{R}$ and $\sigma>0$ stand for the drift and volatility.
It is assumed that $\mu>r$ and the Sharpe ratio is denoted by $\kappa:= \frac{\mu-r}{\sigma} > 0$.

Let $(\pi_t)_{t\geq 0}$ be the amount of wealth that the agent allocates in the risky asset, and let $(c_t)_{t\geq 0}$ represent the consumption rate. The self-financing wealth process $(X_t)_{t\geq 0}$ satisfies
\begin{equation*}
dX_t= \left(rX_t + \pi_t(\mu -r) - c_t\right)dt+\pi_t\sigma dW_t, \quad t\geq 0,
\end{equation*}
with the initial wealth $X_0=x\geq 0$. The control pair $(c, \pi)$ is said to be \textit{admissible} if $c$ is $\mathbb{F}$-predictable and non-negative, $\pi$ is $\mathbb{F}$-progressively measurable, both satisfy the integrability condition $\int_0^{T} (c_t+\pi_t^2)dt<\infty$ a.s. for any $T>0$ as well as the no bankruptcy condition holds that $X_t\geq 0$ a.s. for $t\geq 0$. We use $\mathcal{A}(x)$ to denote the set of admissible controls $(c, \pi)$.

It is assumed in the present paper that the agent is loss averse on relative consumption in the sense that the agent suffers more from a reduction in consumption than would benefit from an increase of the same size. The reference level is chosen as a fraction of the consumption running maximum process $\lambda H_t$,
where $\lam\in(0,1)$ depicts the degree towards the reference,  $H_t:=\max{\{h,\ \sup_{s\leq t} c_s\}}$ denotes the past spending maximum, and $H_0=h\geq 0$ is the initial reference level. The utility maximization problem is defined by
\begin{equation}\label{eq: primalvalue}
u(x, h)=\sup_{(\pi, c)\in\mathcal{A}(x)}\mathbb{E}\left[\int_0^{\infty} e^{-\rho t}U(c_t-\lam H_t)dt\right],
\end{equation}
where $U(\cdot)$ is the S-shaped utility defined in \eqref{eq: utility} with different risk-aversion parameters $\beta_1$ and $\beta_2$ on gains and losses of the relative consumption, and $\rho>0$ is the subjective discount rate to guarantee the convergence of the value function.

Two main challenges in solving \eqref{eq: primalvalue} are the path-dependence of $(H_t)_{t\geq 0}$ on the control $(c_t)_{t\geq 0}$ and the non-concavity of the S-shaped utility  $U(\cdot)$.
As a remedy, we propose to consider the concave envelope of the realization utility on consumption by first assuming the validity of concavification principle (see, for example, Reichlin \cite{Reichlin2013MFE} and Dong and Zheng \cite{DongZheng2020EJOR}).
Later, we plan to characterize the optimal control under the concave envelope function and then verify that the optimal control also attains the value function in the original problem, i.e., the concavification principle indeed holds.
To be precise, for each fixed $h$, let us consider $\tU(c,h)$ as the concave envelope of $U(c-\lam h)$ with respect to the variable $c\in[0, h]$ on a constrained domain.
That is, for each fixed $h\geq 0$, let $\tU(\cdot, h)$ be the smallest concave function on $[0,h]$ such that $\tU(c,h)\geq U(c,h)$ holds for all $c\in[0,h]$.
\subsection{Concave envelope of the realization utility}\label{sec: ce}


To emphasize the concave envelope only with respect to $c\in[0,h]$ while keeping the variable $h$ fixed, let us consider an equivalent bivariate function
\begin{equation*}
U^*(c,h) := U(c-\lam h),
\end{equation*}
on the domain $\{(c,h)\in\mathbb{R}^2: c\in [0,h]\}$.
Define $U_1^*(c,h) := \frac{1}{\beta_1}(c-\lam h)^{\beta_1}$ and $U_2^*(c,h) := -\frac{k}{\beta_2}(\lam h -c)^{\beta_2}$ and denote ${U_1^*}'(c,h) := \frac{\partial U_1^*}{\partial c}(c,h)$ and ${U_2^*}'(c,h) := \frac{\partial U_2^*}{\partial c}(c,h)$. Note that ${U_1^*}'(c,h) \rightarrow +\infty $ as $c\rightarrow(\lambda h)_+$. As ${U_2^*}'(c,h) \rightarrow +\infty$ when $c\rightarrow(\lambda h)_-$, we have two different subcases:




\textbf{Subcase (i)}: If $U_1^*(h, h) + U_2^*(0, h) - h {U_1^*}'(h,h) > 0$, there exists a unique solution $z(h)\in(\lam h,h)$ to the equation
\begin{equation}\label{eq: beta<1,tagent z}
U_1^*(z(h), h) + U_2^*(0, h) - z(h){U_1^*}'(z(h), h) = 0.
\end{equation}
That is, $z(h)$ is the tangent point of the straight line at $(0, -U_2^*(0, h))$ to the curve $U_1^*(c, h)$ for $c\geq \lam h$. Note that $z(h)$ does not admit an explicit expression in this subcase.

\textbf{Subcase (ii)}: If $U_1^*(h, h) + U_2^*(0, h) - h {U_1^*}'(h,h) \leq 0$,
we simply let $z(h)=h$. The concave envelope of $U^*(c,h)$ on $[0, h]$ corresponds to the straight line through two points $(0, U_2^*(0, h))$ and $(h, U_1^*(h,h))$.

\begin{remark}\label{paraRm}
The condition of \textbf{Subcase (ii)} is fulfilled if and only if $h$ and model parameters satisfy one of the following three conditions that
$$\begin{array}{ll}
(\mbox{S1})  &\beta_2>\beta_1\geq1-\lam,~ \mathrm{and}~ h \leq ( \frac{\beta_2(1-\lam)^{\beta_1-1}(\beta_1+\lam-1)}{\beta_1k\lam^{\beta_2}} )^{\frac{1}{\beta_2-\beta_1}},\qquad \\
(\mbox{S2})  &\beta_1\geq1-\lam,~ \beta_1 > \beta_2,~ \mathrm{and}~ h \geq ( \frac{\beta_2(1-\lam)^{\beta_1-1}(\beta_1+\lam-1)}{\beta_1k\lam^{\beta_2}} )^{\frac{1}{\beta_2-\beta_1}}, \\
(\mbox{S3})  &\beta_2=\beta_1\geq 1-\lam, ~1\leq \frac{(1-\lam)^{\beta_1-1}(\beta_1+\lam-1)}{k\lam^{\beta_2}} , ~\mbox{and } h\geq0.
\end{array}$$

\end{remark}

Similar to Dong and Zheng \cite{DongZheng2020EJOR}, we can define the concave envelope of $U^*(c,h)$ for $c\in [0, h]$ by
\begin{equation}\label{eq: ce_beta<1}
\tU(c,h) = \begin{cases}
U_2^*(0, h) + \frac{U_1^*(z(h), h) - U_2^*(0, h)}{z(h)}c, & \mbox{if } 0\leq c < z(h), \\
U_1^*(c, h), & \mbox{if } z(h) \leq c \leq h.
\end{cases}
\end{equation}

\begin{figure}[htbp]\label{fig-1}
\centering
\includegraphics[height=2.2in]{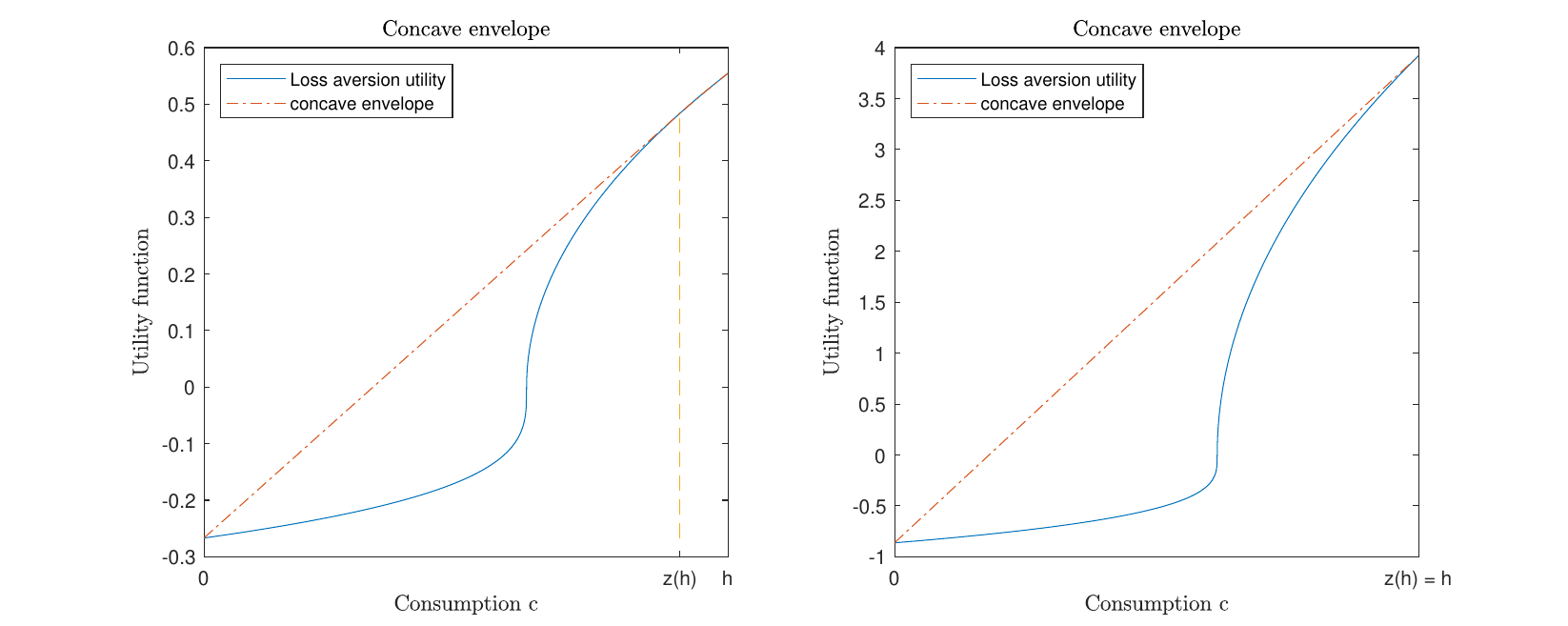}
\caption{\small{Concave envelopes when $0<\beta_2<1$: (left panel) the subcase (i) when $z(h)\neq h$; (right panel) the subcase (ii) when $z(h)= h$.}}
\end{figure}

Figure 1 illustrates two subcases of the concave envelope of the S-shaped utility $U(c,h)$. We stress that the function $\tU(c,h)$ is implicit in $h$ as $z(h)$ is an implicit function in general. To simplify the future presentation, let us also define
\begin{equation}\label{eq: def_w}
w(h) := z(h) - \lam h.
\end{equation}
Hence,  if $z(h)=h$, then $w(h) = (1-\lam)h$, i.e., $z(h) = \lam h + w(h)$.

\subsection{Equivalent problem}
We now consider the auxiliary stochastic control problem
\begin{equation}\label{eq: concavevalue}
\tilde{u}(x, h)=\sup_{(\pi, c)\in\mathcal{A}(x)}\mathbb{E}\left[\int_0^{\infty} e^{-\rho t}\tilde{U}(c_t,H_t)dt\right].
\end{equation}
The equivalence between problems \eqref{eq: primalvalue} and \eqref{eq: concavevalue} is given in the next proposition. Its proof is deferred to Section \ref{sec: proof_prop_samevalue}  after we first establish the verification proof on optimality.

\begin{proposition}[Concavification Principle]\label{prop: samevalue}
Two problems \eqref{eq: primalvalue} and \eqref{eq: concavevalue} admit the same optimal control $(\pi_t^*, c_t^*)$ so that two value functions coincide, i.e., $u(x,h) = \tu(x,h)$ for any $(x,h)\in\mathbb{R}_+\times\mathbb{R}_+$.
\end{proposition}

For problem \eqref{eq: concavevalue}, we can derive the auxiliary HJB variational inequality that
%
\begin{equation}\label{eq: HJB_eqn}
\left\{\begin{array}{rcl}
\underset{c\in [0,h], \pi\in\mathbb{R}}{\sup}\left[ -\rho \tu+\tu_x(rx+\pi(\mu-r)-c)+\frac{1}{2}\sigma^2\pi^2 \tu_{xx}+\tilde{U}(c,h)\right] &=& 0,  \\
\tu_h(x,h)  & \leq & 0,
\end{array}\right.
\end{equation}
for $x\geq 0$ and $h \geq 0$. The free boundary condition $\tilde{u}_h(x, h) = 0$ will be specified later. Our goal is to find the optimal feedback control $c^*(x,h)$ and $\pi^*(x,h)$. If $\tu(x,\cdot)$ is $C^2$ in $x$, the first order condition gives the optimal portfolio in a feedback form by $\pi^{\ast}(x,h)=-\frac{\mu-r}{\sigma^2}\frac{\tu_x}{\tu_{xx}}$. This implies that the HJB variational inequality \eqref{eq: HJB_eqn} can be simplified to
\begin{equation} \label{eq: HJB_main}
\sup_{c\in [0,h]} \left[ \tU(c,h)- c \tu_x \right] -\rho \tu + rx \tu_x -\frac{\kappa^2}{2}\frac{\tu_x^2}{\tu_{xx}} = 0,\ \ \text{and}\ \ \tu_h\leq 0,\ \ \forall x\geq0,h\geq 0.
\end{equation}

\section{Derivation of the Solution}\label{sec: main}
For ease of presentation and technical convenience, we only consider the case that $\rho = r>0$ in the present paper.
Computations in general cases that $\rho\neq r$ and $r=0$ can be conducted similarly. However, some additional sufficient assumptions on model parameters are needed to facilitate the proofs of the verification theorem. Given the implicit concave envelope in \eqref{eq: ce_beta<1}, we can still solve the HJB variational inequality in the analytical form. In particular, we plan to characterize some thresholds (depending on $h$) for the wealth level $x$ such that the auxiliary value function, the optimal portfolio and consumption can be expressed analytically in each region.




Let us first introduce the boundary curves $y_1(h) \geq y_2(h) > y_3(h)$ by
\begin{equation}\label{eq: y_curves_beta<1}
\begin{aligned}
y_1(h) &:= \frac{k(\lam h)^{\beta_2}}{\beta_2z(h)} + \frac{ w(h)^{\beta_1}}{\beta_1z(h)} , \\
y_2(h) &:= \min\bigg( y_1(h), \big((1-\lam)h\big)^{\beta_1-1} \bigg), \\
y_3(h) &:= (1-\lam)^{\beta_1}h^{\beta_1-1},
\end{aligned}
\end{equation}
where $z(h)$ and $w(h)$ are defined in Section \ref{sec: ce}.
Here, $y_1(h) $ and $y_2(h)$ are derivatives of the concave envelope $\tU(c,h)$ at $c=0$ and $c=h$ respectively, which are used to simplify the expression of $\sup_{c\in [0, h]} [ \tU(c,h)- c \tu_x ]$ when the maximum occurs at $c=0$ and $c=h$. We also use $y_3(h)$ to describe the free boundary curve $\tu_h=0$.
Note that if $w(h) \neq (1-\lam)h$, we have $y_1(h) > y_2(h) = ((1-\lam)h)^{\beta_1-1} > (1-\lam)^{\beta_1}h^{\beta_1-1} = y_3(h)$ as $0<\lam<1$ by \eqref{eq: beta<1,tagent z};
on the other hand, if $w(h) = (1-\lam)h$, we have $z(h) = h$, yielding that
$y_1(h) = y_2(h) = \frac{k(\lam h)^{\beta_2}}{\beta_2z(h)} + \frac{ w(h)^{\beta_1}}{\beta_1z(h)} > \frac{w(h)^{\beta_1}}{\beta_1z(h)}
= \frac{1}{\beta_1}(1-\lam)^{\beta_1}h^{\beta_1-1} > (1-\lam)^{\beta_1}h^{\beta_1-1} = y_3(h)$ as $0<\beta_1<1$.

Similar to Deng et al. \cite{DengLiPY2020arXiv}, we can heuristically decompose the domain into several regions based on the first order condition of $c$ and express the HJB equation \eqref{eq: HJB_main} piecewisely. However, the concave envelope of the S-shaped utility complicates the computations here, in which the previous $y_i(h)$ in \eqref{eq: y_curves_beta<1}, $i=1,2,3$, serve as the boundaries of these regions. We can then separate the following regions:

\noindent \emph{Region I}: on the set $\mathcal{R}_1 = \{(x,h) \in \mathbb{R}_+^2: \tu_x(x,h) > y_1(h)\},$
$\tU(c, h) - c\tu_x$ is decreasing in $c$, implying that $c^*=0$ and the HJB equation \eqref{eq: HJB_main} becomes
\begin{equation}\label{eq: HJB_beta<1_1}
-\frac{k}{\beta_2}(\lam h)^{\beta_2} - r\tu + rx\tu_x - \frac{\kappa^2\tu_x^2}{2\tu_{xx}} = 0, ~\mathrm{and}~ \tu_h\leq0.
\end{equation}

\noindent \emph{Region II}: on the set
$\mathcal{R}_2 = \{(x,h)\in \mathbb{R}_+^2 : y_2(h) \leq \tu_x(x,h) \leq y_1(h)\},$
$\tU(c, h) - c\tu_x$ is increasing on $[0, z(h)]$ and concave on $[z(h), h]$, implying that $c^*=\lam h + \tu_x^{\frac1{\beta_1-1}} \geq z(h)$ and the HJB equation \eqref{eq: HJB_main} becomes
\begin{equation}\label{eq: HJB_beta<1_2}
\frac{1-\beta_1}{\beta_1}\tu_x^{\frac{\beta_1}{\beta_1-1}} - \lam h \tu_x
- r\tu + rx\tu_x - \frac{\kappa^2\tu_x^2}{2\tu_{xx}} = 0, ~\mathrm{and}~ \tu_h\leq0.
\end{equation}

\noindent \emph{Region III}: on the set $\mathcal{R}_3  = \{(x,h) \in \mathbb{R}_+ \times \mathbb{R}_+ : \tu_x(x,h) < y_2(h)\}$, $\tU(c) - c\tu_x$ is increasing in $c$ on $[0,h]$, implying that $c^*=h$.
To distinguish whether the optimal consumption $c_t^*$ updates the past maximum process $H_t^*$ in this region, one can heuristically substitute $h=c$ in \eqref{eq: HJB_main} and apply the first order condition to $\tU(c, c) - c\tu_x$ with respect to $c$ and derive the auxiliary singular control $\hat{c}(x) := \tu_x^{\frac1{\beta_1-1}}(1-\lam)^{-\frac{\beta_1}{\beta_1-1}}$.
We then need to split \emph{Region III} further into three subsets:

\noindent \emph{Region III-(i)}: on the set
$\mathcal{D}_1= \{(x,h) \in \mathbb{R}_+ \times \mathbb{R}_+ : y_3(h) < \tu_x < y_2(h)\}$,
it is easy to see a contradiction that $\hat{c}(x) < h$, and therefore the optimal consumption $c_t^*$ does not equal to $\hat{c}$ and
we should follow the previous feedback form $c^*(x,h) = h$, in which $h$ is a previously attained maximum level. The HJB variational inequality is written as
\begin{equation}\label{eq: HJB_beta<1_3}
\frac1{\beta_1}((1-\lam)h)^{\beta_1} - h\tu_x
- r\tu + rx\tu_x - \frac{\kappa^2\tu_x^2}{2\tu_{xx}} = 0, ~\mathrm{and}~ \tu_h\leq0.
\end{equation}

\noindent \emph{Region III-(ii)}: on the set $D_2 := \{ (x,h)\in\mathbb{R}_+ \times \mathbb{R}_+ : \tu_x(x,h) = y_3(h)\}$, we get $\hat{c}(x) = h$ and the feedback optimal consumption is $c^*(x,h) = \tu_x^{\frac{1}{\beta_1-1}}(1-\lam)^{-\frac{\beta_1}{\beta_1-1}} = h$.
This corresponds to the singular control $c_t^*$ that creates a new peak for the whole path that $H_t^* > H_s^*$ for any $s<t$.  We then impose the free boundary condition $\tu_h(x,h) = 0$ in this region{\cblue,} and the HJB equation follows the same PDE in \eqref{eq: HJB_beta<1_3}.

\noindent \emph{Region III-(iii)}: on the set $\mathcal{D}_3 := \{(x,h) \in\mathbb{R}_+ \times \mathbb{R}_+ : \tu_x(x,h) < y_3(h)\}$, we get $\hat{c}(x) > h$.
The optimal consumption is again a singular control $c^*(x) = \tu_x^{\frac{1}{\beta_1-1}}(1-\lam)^{-\frac{\beta_1}{\beta_1-1}} > h$, which pulls the associated $H_{t-}^*$ upward to the new value $\tu_x(X_t^*, H_t^*)^{\frac{1}{\beta_1-1}}(1-\lam)^{-\frac{\beta_1}{\beta_1-1}}$, in which $\tu(x,h)$ is the solution of the HJB equation on the set $\mathcal{D}_2$.
This suggests that for any given initial value $(x,h)$ in the set $\mathcal{D}_3$, the feedback control $c^*(x,h)$ pushes the value function jumping immediately to the point $(x, \hat{h})$ on the boundary set $\mathcal{D}_2$.

In summary, it is sufficient to consider the effective domain defined by
\begin{equation}\label{eq: effective domain}
\begin{aligned}
\mathcal{C} &:= \{(x,h) \in \mathbb{R}_+ \times \mathbb{R}_+ :
\tu_x(x,h) \geq y_3(h)\} \\
&= \mathcal{R}_1 \cup \mathcal{R}_2 \cup \mathcal{D}_1 \cup \mathcal{D}_2 \subset \mathbb{R}_+^2,
\end{aligned}\end{equation}
and $(x,h) \in \mathcal{D}_3$ can only occur at the initial time $t=0$.

Therefore, the HJB variational inequality \eqref{eq: HJB_main} can be written to
\begin{equation}\label{eq: HJB_beta<1}
\begin{aligned}
-r\tu + rx\tu_x - \frac{\kappa^2\tu_x^2}{2\tu_{xx}} = -V(u_x, h) , ~\mathrm{and}~ \tu_h\leq0, \\
\tu_h = 0, ~\mbox{if}~ \tu_x = y_3(h),
\end{aligned}\end{equation}
where
$$
V(q, h) := \begin{cases}
-\frac{k}{\beta_2}(\lam h)^{\beta_2}, & \mbox{if } q > y_1(h),  \\
-\frac{\beta_1-1}{\beta_1}q^{\frac{\beta_1}{\beta_1-1}} - \lam h q, & \mbox{if } y_2(h) \leq q \leq y_1(h), \\
\frac1{\beta_1}((1-\lam)h)^{\beta_1} - hq, & \mbox{if } y_3(h) \leq q < y_2(h).
\end{cases}
$$
To solve the above equation, some boundary conditions are also needed.
First, to guarantee the global regularity of the solution, we need to impose smooth-fit conditions along two free boundaries that $\tu_x(x,h) = y_1(h)$ and $\tu_x(x,h) = y_2(h)$.
Next, if we start with 0 initial wealth, to avoid bankruptcy, the optimal investment and the consumption rate should be 0 at all times.
Therefore, we have that
\begin{equation}\label{eq: boundary_x->0}
\lim\limits_{x\rightarrow0} \frac{\tu_x(x,h)}{\tu_{xx}(x,h)} = 0
~~\mathrm{and}~~
\lim\limits_{x\rightarrow0} \tu(x, h) =  \int_0^{+\infty} -\frac{k}{\beta_2}(\lam h)^{\beta_2}e^{-rt}dt = -\frac{k}{r\beta_2}(\lam h)^{\beta_2}.
\end{equation}
On the other hand, when the initial wealth tends to infinity, one can consume as much as possible, leading to an infinitely large consumption rate.
In addition, a small variation of initial wealth only leads to a negligible change of the value function.
It follows that
\begin{equation}\label{eq: boundary_x->infty}
\lim\limits_{x\rightarrow+\infty} \tu(x,h) = +\infty
~~\mathrm{and}~~
\lim\limits_{x\rightarrow+\infty} \tu_x(x,h) = 0.
\end{equation}

We also note that, as the initial value $x$ is large enough, we have $(x,h) \in \mathcal{D}_2$ and thus $c^*(x) = \tu_x(x, h)^{\frac{1}{\beta_1-1}}(1-\lam)^{-\frac{\beta_1}{\beta_1-1}}.$
Intuitively, our problem is similar to the Merton problem \cite{Mert1969RES} along the free boundary $\mathcal{D}_2$, in which the optimal consumption is asymptotically proportional to the wealth. Therefore, we expect to have that
\begin{equation}\label{eq: boundary_x->infty_2}
\lim\limits_{\substack{x\rightarrow+\infty \\ (x,h)\in \mathcal{D}_2}}
\frac{\tu_x(x, h)^{\frac{1}{\beta_1-1}}}{x} = c_\infty,
\end{equation}
for some constant $c_\infty >0$.
This condition will be verified later in Corollary \ref{cor: asy_infty_wealth}.

To tackle the nonlinear HJB equation \eqref{eq: HJB_beta<1}, we employ the dual transform only with respect to the variable $x$ and treat the variable $h$ as a parameter; see similar dual transform arguments in Deng et al. \cite{DengLiPY2020arXiv} and Bo et al. \cite{BLY2021SICON}. That is,
we consider $v(y,h) := \sup_{x\geq0}\{\tu(x,h) - xy\}$, $y\geq y_3(h)$.
For a given $(x,h)\in\mathcal{C}$, define the variable $y = \tu_x(x,h)$ and it holds that $\tu(x,h) = v(y,h) + xy$.
We can further deduce that
$x = -v_y(y,h)$, $\tu(x,h) = v(y,h) - yv_y(y,h)$, and $\tu_{xx}(x,h) = -\frac{1}{v_{yy}(y,h)}$. The nonlinear ODE \eqref{eq: HJB_beta<1} can be linearized to
\begin{equation}\label{eq: LODE_beta<1}
\frac{\kappa^2}{2}y^2 v_{yy} - rv =  -V(y,h),
\end{equation}
and the free boundary condition is transformed to $y = y_3(h)$. As $h$ can be regarded as a parameter, we can study the above equation as an ODE problem of the variable $y$. Based on the dual transform, the boundary conditions \eqref{eq: boundary_x->infty}  can be written as
\begin{equation}\label{eq: boundary_dual_y->0}
\lim\limits_{y\rightarrow0} v_y(y,h) = -\infty,  ~~ \mathrm{and} ~~
\lim\limits_{y\rightarrow0} (v(y,h)-yv_y(y,h)) = +\infty.
\end{equation}
The boundary condition \eqref{eq: boundary_x->infty_2} becomes
\begin{equation}\label{eq: boundary_dual_y->0_2}
\lim\limits_{y\rightarrow0 }
 \frac{y^{\frac{1}{\beta_1-1}}}{v_y(y,h)} = -c_\infty,
\end{equation}
along the boundary curve $y_3(h) = (1-\lam)^{\beta_1}h^{\beta_1-1}$.
The boundary condition \eqref{eq: boundary_x->0} is equivalent to
\begin{equation}\label{eq: boundary_dual_y->infty}
yv_{yy}(y,h) \rightarrow 0
~~\mathrm{and}~~
v(y,h) - yv_{y}(y,h) \rightarrow -\frac{k}{r\beta_2}(\lam h)^{\beta_2}
~~\mathrm{as}~~
v_y(y,h) \rightarrow 0.
\end{equation}
It holds by the dual transform that $v_y(y,h) = -x$, and one can derive that
$\tu_h(x,h)= v_h(y,h) + (v_y(y,h) + x)\frac{dy(h)}{dh}= v_h(y,h)$. The free boundary condition \eqref{eq: HJB_beta<1} is translated to
\begin{equation}\label{eq: free_boundary_dual}
 v_h(y,h)  = 0 ~~\mathrm{for}~~ y = y_3(h).
\end{equation}

Although the dual ODE problem looks similar to the one in Deng et al. \cite{DengLiPY2020arXiv}, we emphasize that the boundary curves $y_1(h)$ and $y_2(h)$ are implicit functions of $h$ within which the implicit function z(h) lies. As a result, it becomes more complicated to apply smooth-fit conditions to derive the solution analytically and to prove the verification theorem. It is inevitable that all coefficient functions (in terms of $h$) in the solution involve $z(h)$. In particular, the following assumption on model parameters is needed, which is used in showing that the obtained solution $v(y,h)$ is convex in $y$, and in the verification proof of the optimal control. \\

\noindent \textbf{Assumption (A1)} $\beta_j < -\frac{r_2}{r_1}$, $j=1,2$, where $r_1>1$ and $r_2<0$ are two roots to the equation $\eta^2 - \eta - \frac{2r}{\kappa^2} = 0$.

Note that $\beta_j < -\frac{r_2}{r_1}$ implies that $\gamma_j = \frac{\beta_j}{\beta_j-1} > r_2$, $r_1\beta_j+r_2 = (\gamma_j-r_2)(\beta_j-1) < 0$, for $j=1,2$.
\vspace{0.1in}



\begin{proposition}\label{prop: dual_solution_beta<1}
Let \textbf{Assumption (A1)} hold. Under boundary conditions \eqref{eq: boundary_dual_y->0}, \eqref{eq: boundary_dual_y->0_2}, \eqref{eq: boundary_dual_y->infty}, the free boundary condition \eqref{eq: free_boundary_dual}, and the smooth-fit conditions with respect to $y$ along $y = y_1(h)$ and $y = y_2(h)$, ODE \eqref{eq: LODE_beta<1} in $\{y\in\mathbb{R}: y\geq y_3(h)\}$ admits the unique solution that
\begin{equation}\label{eq: dual_solu_beta<1}
v(y,h) = \begin{cases}
C_2(h)y^{r_2} - \frac{k}{r\beta_2}(\lam h)^{\beta_2}, & \mbox{if } y> y_1(h),  \\
C_3(h)y^{r_1} + C_4(h)y^{r_2} + \frac{2}{\kappa^2\gamma_1(\gamma_1-r_1)(\gamma_1-r_2)}y^{\gamma_1} - \frac{\lam h}{r}y, & \mbox{if } y_2(h) \leq y \leq y_1(h), \\
C_5(h)y^{r_1} + C_6(h)y^{r_2} + \frac{1}{r\beta_1}((1-\lam)h)^{\beta_1} - \frac{h}{r}y, & \mbox{if } y_3(h) \leq y < y_2(h),  \\
\end{cases}
\end{equation}
where $\gamma_1 = \frac{\beta_1}{\beta_1-1} < 0$,
$w(h)$ is defined in \eqref{eq: def_w},
$r_1>1$ and $r_2<0$ are given in \textbf{Assumption (A1)},
$y_1(h), y_2(h)$ and $y_3(h)$ are given in \eqref{eq: y_curves_beta<1},
and functions $C_i(h)$, $i=2, \ldots, 6$, are defined by
\begin{align}\label{paraC}
C_2(h) &:= C_4(h) + \frac{y_1(h)^{-r_2}}{r(r_1-r_2)}
\bigg( \frac{kr_1}{\beta_2}(\lam h)^{\beta_2} + \frac{r_1r_2}{\gamma_1(\gamma_1-r_2)}y_1(h)^{\gamma_1}+ \lam h r_2 y_1(h) \bigg),\nonumber\\
C_3(h) &:= \frac{y_1(h)^{-r_1}}{r(r_1-r_2)}
\bigg( \frac{kr_2}{\beta_2}(\lam h)^{\beta_2} + \frac{r_1r_2}{\gamma_1(\gamma_1-r_1)}y_1(h)^{\gamma_1}+ \lam h r_1 y_1(h) \bigg),\nonumber\\
C_4(h) &:= C_6(h) + \frac{y_2(h)^{-r_2}}{r(r_1-r_2)}\bigg( \frac{r_1}{\beta_1}((1-\lam)h)^{\beta_1} - \frac{r_1r_2}{\gamma_1(\gamma_1-r_2)}y_2(h)^{\gamma_1} + (1-\lam) h r_2y_2(h) \bigg),\nonumber\\
C_5(h) &:= C_3(h) + \frac{y_2(h)^{-r_1}}{r(r_1-r_2)}\bigg( \frac{r_2}{\beta_1}((1-\lam)h)^{\beta_1} - \frac{r_1r_2}{\gamma_1(\gamma_1-r_1)}y_2(h)^{\gamma_1} + (1-\lam) h r_1y_2(h) \bigg),\nonumber\\
C_6(h) &:= \int_{h}^{+\infty} (1-\lam)^{(r_1-r_2)\beta_1} C'_5(s) s^{(r_1-r_2)(\beta_1-1)}ds.
\end{align}
\end{proposition}

\begin{remark}
Note that all $C_i(h)$, $i=2,\ldots, 6$, are implicit functions of $h$. In particular, $C_2(h)$, $C_4(h)$ and $C_6(h)$ are written in the integral form. $C_3(h)$ and $C_5(h)$ are written in terms of implicit functions $y_1(h)$ and $y_2(h)$. Some technical efforts are needed to handle these semi-analytical functions in the later verification proof.
\end{remark}

\begin{theorem}[Verification Theorem]\label{thm: beta<1}
Let $(x,h) \in \mathcal{C}$, $h\in\mathbb{R}$ and $0<\lam <1$, where $x\geq0$ stands for the initial wealth, $h\geq0$ is the initial reference level, and $\mathcal{C}$ is the effective domain \eqref{eq: effective domain}.
Let \textbf{Assumption (A1)} hold. For $(y,h)\in\{(y,h)\in\mathbb{R}_+^2: y \geq y_3(h)\}$, define the feedback functions that
\begin{equation}\label{eq: c_y_beta<1}
c^\dag(y,h) = \begin{cases}
0, & \mbox{if } y > y_1(h),  \\
\lam h + y^{\frac1{\beta_1-1}}, & \mbox{if } y_2(h) \leq y \leq y_1(h), \\
h, & \mbox{if } y_3(h) < y < y_2(h), \\
y^{\frac{1}{\beta_1-1}}(1-\lam)^{-\frac{\beta_1}{\beta_1-1}}, & \mbox{if } y = y_3(h),
\end{cases}
\end{equation}
and
\begin{equation}\label{eq: pi_y_beta<1}
\begin{aligned}
&\pi^\dag(y,h) = \frac{\mu-r}{\sigma^2}y v_{yy}(y,h) \\
= & \frac{\mu - r}{\sigma^2}
\begin{cases}
\frac{2r}{\kappa^2} C_2(h)y^{r_2-1}, & \mbox{if } y > y_1(h), \\
\frac{2r}{\kappa^2} C_3(h)y^{r_1-1} +  \frac{2r}{\kappa^2} C_4(h)y^{r_2-1}
+\frac{2(\gamma_1-1)}{\kappa^2(\gamma_1-r_1)(\gamma_1-r_2)}y^{\gamma_1-1},
& \mbox{if } y_2(h) \leq y \leq y_1(h),\\
\frac{2r}{\kappa^2} C_5(h)y^{r_1-1} +  \frac{2r}{\kappa^2} C_6(h)y^{r_2-1} , & \mbox{if } y_3(h) \leq y < y_2(h).
\end{cases}
\end{aligned}
\end{equation}
We consider the process $Y_t := y^* e^{rt}M_t$, where $M_t := e^{-(r+\frac{\kappa^2}{2})t - \kappa W_t}$ is the discounted rate state price density process, and $y^* = y^*(x,h)$ is the unique solution to the budget constraint $\mathbb{E}[ \int_0^\infty c^\dag(Y_t(y), H_t^\dag(y))$ $M_tdt ] = x$ with
$$
H_t^\dag(y) := h\vee \sup\limits_{s\leq t} c^\dag(Y_s(y), H_s^\dag(y))
=h\vee \bigg((1-\lam)^{-\frac{\beta_1}{\beta_1-1}} (\inf_{s\leq t}Y_s(y))^{\frac{1}{\beta_1-1}} \bigg){\cblue.}
$$
The value function $\tu(x,h)$ can be attained by employing the optimal consumption and portfolio strategies in the feedback form that $c_t^* = c^\dag(Y_t^*, H_t^*)$ and $\pi_t^* = \pi^\dag(Y_t^*, H_t^*)$ for all $t\geq0$, where $Y_t^*:=Y_t(y^*)$ and $H_t^*=H_t^\dag(y^*)$.
\end{theorem}

\begin{remark}\label{remark: predictable}
Note that the optimal consumption $c_t^*$ has a jump when $Y_t^*=y_1(H_t^*)$ and $c_t^*>\lambda H_t^*$ whenever $c_t^*>0$. Meanwhile, we note that the running maximum process $H_t^*$ still has continuous paths for $t>0$. Indeed, from the feedback form, $c_{t-}^*$ jumps only when $c_{t-}^*<H_t^*$ and we also have that $c_t^*\leq H_t^*$ after the jump, i.e., the jump can never increase $H_t^*$. Therefore, both $X_t^*$ and $H_t^*$ still have continuous paths.
\end{remark}

By the dual representation, we have that $x = g(\cdot, h) := -v_y(\cdot,h)$.
Define $f(\cdot, h)$ as the inverse of $g(\cdot, h)$, then $\tu(x,h) = v\circ(f(x,h), h) + xf(x,h)$.
Note that the function $f$ should have a piecewise form across different regions.
The invertibility of the map $x\mapsto g(x,h)$ is guaranteed by the next lemma.

\begin{lemma}\label{lemma: beta<1}
Let \textbf{Assumption (A1)} hold. The function $v(y,h)$ is convex in all regions so that the inverse Legendre transform $\tu(x,h) = \inf_{y\geq y_3(h)} [v(y,h) + xy]$ is well defined.
Moreover, it implies that the feedback optimal portfolio $\pi^*(y,h) > 0$.
\end{lemma}
\begin{proof}
The proof is given in Section \ref{sec: proof_lemma_beta<1}.
\end{proof}
Thanks to Lemma \ref{lemma: beta<1}, we can apply the inverse Legendre transform to the solution $v(y, h)$ in \eqref{eq: dual_solu_beta<1}.
Similar to Section 3.1 in \cite{DengLiPY2020arXiv}, we can derive the following three boundary curves $x_\mathrm{zero}(h),$ $x_\mathrm{aggr}(h),$ and $x_\mathrm{lavs}(h)$ that
\begin{equation}\label{eq: boundary_beta<1}
\begin{aligned}
x_{\mathrm{zero}}(h) &:= -y_1(h)^{r_2-1}C_2(h)r_2,\\
x_{\mathrm{aggr}}(h) &:= -C_3(h)r_1y_2(h)^{r_1-1} - C_4(h)r_2y_2(h)^{r_2-1} - \frac{2}{\kappa^2(\gamma_1-r_1)(\gamma_1-r_2)}y_2(h)^{\gamma_1-1} + \frac{\lam h}{r},\\
x_{\mathrm{lavs}}(h) &:= -C_5(h)r_1y_3(h)^{r_1-1}- C_6(h)r_2y_3(h)^{r_2-1} + \frac{h}{r},
\end{aligned}
\end{equation}
and it holds that the feedback function of the optimal consumption satisfies: (i) $c^*(x,h) = 0$ when $0<x\leq x_\mathrm{zero}(h)$;
(ii) $0 < c^*(x,h) \leq h$ when $x_\mathrm{zero}(h) \leq x \leq x_\mathrm{aggr}(h)$; (iii) $c^*(x,h)=h$ when $x_\mathrm{aggr}(h) < x \leq x_\mathrm{lavs}(h)$.
In particular, the condition $u_x(x,h) \geq y_3(h)$ in the effective domain can be explicitly expressed as $x\leq x_\mathrm{lavs}(h)$.
Moreover, the following inverse function is well defined that
\begin{equation}\label{eq: tildeh}
\tilde{h}(x) := (x_\mathrm{lavs})^{-1}(x), x\geq0.
\end{equation}
Along the boundary $x=x_\mathrm{lavs}(h)$, the feedback form of the optimal consumption in \eqref{eq: c_beta<1} is $c^*(x,h)=(1-\lam)^{-\frac{\beta_1}{\beta_1-1}}f(x,\tilde{h}(x))^{-\frac1{\beta_1-1}}$. Using the dual relationship and Proposition \ref{prop: dual_solution_beta<1}, the function $f$ can be implicitly determined as follows:
\begin{itemize}
\item[(i)] If 
$x < x_\mathrm{zero}(h)$, Lemma \ref{lemma: beta<1} implies that $v_y(y,h)$ is strictly increasing in $y$ and $f(x,h) = f_1(x,h)$ can be uniquely determined by
\begin{equation*}\label{eq: f_1_beta<1}
x = -C_2(h)r_2(f_1(x,h))^{r_2-1}.
\end{equation*}
\item[(ii)] If 
$x_\mathrm{zero}(h) \leq x \leq x_\mathrm{aggr}(h)$, Lemma \ref{lemma: beta<1} implies that $v_y(y,h)$ is strictly increasing in $y$ and $f(x,h) = f_2(x,h)$ can be uniquely determined by
\begin{equation}\label{eq: f_2_beta<1}
\begin{aligned}
x = &-C_3(h)r_1(f_2(x,h))^{r_1-1} - C_4(h)r_2(f_2(x,h))^{r_2-1}- \frac{2}{\kappa^2(\gamma_1-r_1)(\gamma_1-r_2)}(f_2(x,h))^{\gamma_1-1} + \frac{\lam h}{r}.
\end{aligned}\end{equation}
\item[(iii)] If 
$x_\mathrm{aggr}(h) < x \leq x_\mathrm{lavs}(h)$, Lemma \ref{lemma: beta<1} implies that $v_y(y,h)$ is strictly increasing in $y$ and $f(x,h) = f_3(x,h)$ can be uniquely determined by
\begin{equation}\label{eq: f_3_beta<1}
x = -C_5(h)r_1(f_3(x,h))^{r_1-1} - C_6(h)r_2(f_3(x,h))^{r_2-1} + \frac{h}{r}.
\end{equation}
\end{itemize}

\begin{corollary}\label{cor: beta<1}
For $(x,h) \in \mathcal{C}$, $0<\lam<1$, $\beta_1<1$ and $\beta_2<1$, under \textbf{Assumption (A1)}, we define the piecewise function
$$
f(x,h) = \begin{cases}
           \left( \frac{-x}{C_2(h)r_2} \right)^{\frac1{r_2-1}}, & \mbox{if }x< x_\mathrm{zero}(h),  \\
           f_2(x,h), & \mbox{if }x_\mathrm{zero}(h) \leq x \leq x_\mathrm{aggr}(h),  \\
           f_3(x,h), & \mbox{if }x_\mathrm{aggr}(h) < x \leq x_\mathrm{lavs}(h),
         \end{cases}
$$
where $f_2(x,h)$ and $f_3(x,h)$ are defined in \eqref{eq: f_2_beta<1} and \eqref{eq: f_3_beta<1}, respectively.

The value function $\tu(x,h)$ in \eqref{eq: concavevalue} can be written as
\begin{equation*}\label{eq: u_beta<1}
\tu(x,h) = \begin{cases}
C_2(h)(f(x,h))^{r_2} - \frac{k}{r\beta_2}(\lam h)^{\beta_2} + xf(x,h), & \mbox{if } x < x_\mathrm{zero}(h), \\
\begin{aligned}
&C_3(h)(f(x,h))^{r_1} + C_4(h)(f(x,h))^{r_2} \\
&+ \frac{2}{\kappa^2\gamma_1(\gamma_1-r_1)(\gamma_1-r_2)}(f(x,h))^{\gamma_1} - \frac{\lam h}{r}f(x,h) + xf(x,h),
\end{aligned}
& \mbox{if } x_\mathrm{zero}(h) \leq x \leq x_\mathrm{aggr}(h), \\
\begin{aligned}
&C_5(h)(f(x,h))^{r_1} + C_6(h)(f(x,h))^{r_2} \\ &+\frac{1}{r\beta_1}((1-\lam)h)^{\beta_1} - \frac{h}{r}f(x,h) + xf(x,h),
\end{aligned}
& \mbox{if } x_\mathrm{aggr}(h) < x \leq x_\mathrm{lavs}(h),
\end{cases}
\end{equation*}
where the free boundaries $x_\mathrm{zero}(h)$, $x_\mathrm{aggr}(h)$, and $x_\mathrm{lavs}(h)$ are given explicitly in \eqref{eq: boundary_beta<1}.
The feedback optimal consumption and portfolio can be expressed in terms of primal variables $(x,h)$ that
\begin{equation}\label{eq: c_beta<1}
c^*(x,h) = \begin{cases}
0, & \mbox{if } x < x_\mathrm{zero}(h), \\
\lam h + (f(x,h))^{\frac1{\beta_1-1}}, & \mbox{if } x_\mathrm{zero}(h) \leq x \leq x_\mathrm{aggr}(h) , \\
h, & \mbox{if } x_\mathrm{aggr}(h) < x < x_\mathrm{lavs}(h), \\
(1-\lam)^{-\frac{\beta_1}{\beta_1-1}}f(x,\tilde{h}(x))^{-\frac1{\beta_1-1}}, & \mbox{if } x = x_\mathrm{lavs}(h) ,
\end{cases}
\end{equation}
where $\tilde{h}(x)$ is given in \eqref{eq: tildeh}, and
\begin{equation}\label{eq: pi_beta<1}
\begin{aligned}
&\pi^*(x,h) \\
=&\frac{\mu-r}{\sigma^2}
\begin{cases}
(1-r_2)x, & \mbox{if } x < x_\mathrm{zero}(h), \\
\begin{aligned}
&\bigg(\frac{2r}{\kappa^2} C_3(h)f^{r_1-1}(x,h) +  \frac{2r}{\kappa^2} C_4(h)f^{r_2-1}(x,h) \\
&+ \frac{2(\gamma_1-1)}{\kappa^2(\gamma_1-r_1)(\gamma_1-r_2)}f^{\gamma_1-1}(x,h)\bigg),
\end{aligned}
& \mbox{if } x_\mathrm{zero}(h) \leq x \leq x_\mathrm{aggr}(h), \\
\frac{2r}{\kappa^2} C_5(h)f^{r_1-1}(x,h) +  \frac{2r}{\kappa^2}C_6(h)f^{r_2-1}(x,h) ,
& \mbox{if } x_\mathrm{aggr}(h)< x\leq x_\mathrm{lavs}(h).
\end{cases}
\end{aligned}
\end{equation}

Moreover, for any initial value $X_0^*, H_0^* = (x,h) \in \mathcal{C}$, the stochastic differential equation
\begin{equation}\label{eq: SDE}
dX_t^* = rX_t^*dt + \pi^*(\mu-r)dt + \pi^*\sigma dW_t - c^*dt,
\end{equation}
has a unique strong solution given the optimal feedback control $(c^*, \pi^*)$ as above.
\end{corollary}

\begin{proof}
The proof is given in Section \ref{sec:proofC3.1}.
\end{proof}

\section{Properties of Optimal Controls}\label{sec: long_run}
First, comparing with the main results in the standard Merton's problem with power utility (see \cite{Mert1969RES}), our optimal feedback controls $\pi^*(x,h)$ and $c^*(x,h)$ are fundamentally different, as they are expressed as the piecewise implicit nonlinear functions of both variables $x$ and $h$. In particular, our optimal consumption process exhibits jumps when the wealth level crosses the threshold $x_\mathrm{zero}(h)$. The more complicated solution structure is rooted in the path-dependent reference process $H_t$ inside the utility and the S-shaped utility accounting for the loss aversion.

Moreover, based on Corollary \eqref{cor: beta<1}, we can show some asymptotic results on the optimal consumption-wealth ratio $c_t^*/X_t^*$ and the optimal portfolio-wealth ratio $\pi_t^*/X_t^*$, whose proof is given in Section \ref{sec: proof_asy_infty_wealth}.

\begin{corollary}\label{cor: asy_infty_wealth}
As $x\leq x_{\text{lavs}}(h)$, the asymptotic behavior of large wealth $x\rightarrow+\infty$ is equivalent to $\lim_{h\rightarrow+\infty}x_{\text{lavs}}(h)=+\infty$. We then have that
\begin{align*}
\lim_{h\rightarrow+\infty} \frac{c^*(x_{\text{lavs}}(h),h)}{x_{\text{lavs}}(h)}=L_1,\quad\quad \lim_{h\rightarrow+\infty} \frac{\pi^*(x_{\text{lavs}}(h),h)}{x_{\text{lavs}}(h)}=L_2{\cred,}
\end{align*}
for some constants $L_1$ and $L_2$. In addition, as $\lam\rightarrow0$, two limits $L_1$ and $L_2$ coincide with the asymptotic results in the infinite-horizon Merton's problem \cite{Mert1969RES} with power utility $U^*(x) = \frac{1}{\beta_1}x^{\beta_1}$. As a result, the boundary conditions \eqref{eq: boundary_x->infty_2} and \eqref{eq: boundary_dual_y->0_2} hold in our problem.
\end{corollary}

\begin{remark}
As the wealth level gets sufficiently large, both the optimal consumption and the optimal portfolio amount are asymptotically proportional to the wealth level that $c^*_t\approx L_1 X^*_t$ and $\pi^*_t\approx L_2X^*_t$, in a similar fashion to the asymptotic results in the standard Merton's problem with the power utility. However, it is important to note that our asymptotic limits differ significantly from the ones in the Merton's problem, which now sensitively depends on the reference degree parameter $\lambda$ and risk aversion parameters $\beta_1, \beta_2, k$ from the S-shaped utility. One can see that, even when the agent's wealth level is very high, the impacts from the reference level $\lambda H_t$ and the loss aversion preference do not fade out in our model due to the fact that the large consumption rate $c^*_t$ also lifts up the reference $\lambda H_t^*$ to a new high level. Only in the extreme case when the reference degree $\lam\rightarrow0$, i.e., there is no reference process, our asymptotic results coincide with the ones in the standard Merton's problem under the power utility.

\end{remark}

Next, we can characterize the average fraction of time that the agent expects to stay in each region.
\begin{corollary}\label{thm: long_run}
The following properties hold:
\begin{enumerate}
\item The long-run fraction of time that the agent stays in the region $\{x_\mathrm{aggr}(H^*_t)\leq X^*_t \leq x_\mathrm{lavs}(H^*_t)\}$ equals $\lim\limits_{h\rightarrow+\infty} \frac{y_2(h)}{y_1(h)}$.
\item The long-run fraction of time that the agent stays in the region $\{0\leq X^*_t\leq x_\mathrm{zero}(H^*_t)\}$ equals $1-\lim\limits_{h\rightarrow+\infty} \frac{y_3(h)}{y_1(h)}$.

\item Starting from $(x,h)\in\{(x,h): x\in(x_\mathrm{zero}(h), x_\mathrm{lavs}(h)]\}$, let us consider the first hitting time of zero consumption that $\tau_\mathrm{zero} := \inf\{t\geq0, X_t = x_\mathrm{zero}(H_t)\}$. We have that
    $$
    \E[\tau_\mathrm{zero}] = \overline{C}_1(h)f(x,h)^{2} + \overline{C}_2(h) + \frac{\log f(x,h)}{\kappa^2},
    $$
    where $\overline{C}_1(h)$ and  $\overline{C}_2(h)$ satisfy:
    $$
    \begin{aligned}
    &\overline{C}_1(h)y_1(h)^2 + \overline{C}_2(h) + \frac{\log y_1(h)}{\kappa^2} = 0, \\
    &\overline{C}_1'(h)y_3(h)^2 + \overline{C}_2'(h) = 0.
    \end{aligned}
    $$
\item Starting from $(x,h)\in\{(x,h): x\in[x_\mathrm{zero}(h), x_\mathrm{lavs}(h))\}$, we define the first hitting time to update the historical consumption maximum $\tau_\mathrm{lavs} := \inf\{t\geq0: X_t = x_\mathrm{lavs}(H_t)\}$. We have that
    $$
    \E[\tau_\mathrm{lavs}] = \frac{2}{\kappa^2}\log\bigg(\frac{f(x,h)h^{1-\beta_1}}{(1-\lam)^{\beta_1}}\bigg).
    $$

\end{enumerate}
\end{corollary}

We next present some numerical examples of the thresholds and the optimal feedback functions and discuss some financial implications.

\begin{figure}[htb]
\centering
\includegraphics[width=0.85\textwidth]{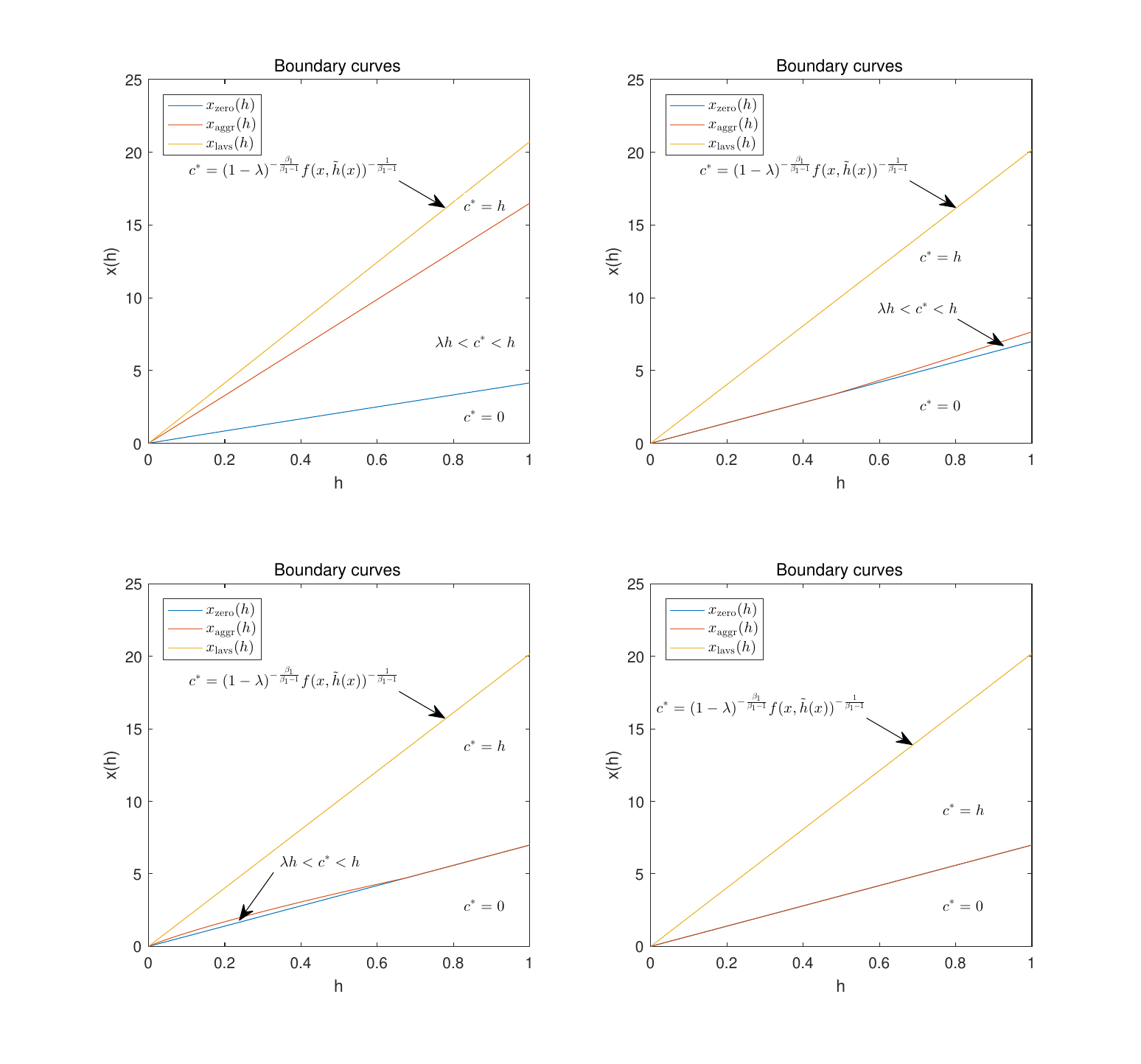}
\vspace{-0.3in}
\caption{{\small Four cases of boundary curves caused by different parameters}}
\label{fig: pic_free_boundary_beta<1_1}
\end{figure}

We first plot in Figure \ref{fig: pic_free_boundary_beta<1_1} the boundary curves $x_\mathrm{zero}(h)$, $x_\mathrm{aggr}(h)$ and $x_\mathrm{lavs}(h)$ as functions of $h$, separating the regions for different feedback forms of the optimal consumption.
First, comparing with Figure 1 in Deng et al. \cite{DengLiPY2020arXiv}, it is interesting to note that we need to distinguish four different cases in total depending on whether two boundary curves $x_\mathrm{zero}(h)$ and $x_\mathrm{aggr}(h)$ may coincide or not. To be more precise, we know by definition that $x_\mathrm{zero}(h) = x_\mathrm{aggr}(h)$ if and only if $y_1(h) = y_2(h)$, where $y_1(h)$ and $y_2(h)$ are given in \eqref{eq: y_curves_beta<1}.
In view of Remark \ref{paraRm}, $y_1(h)=y_2(h)$ in three different scenarios.
The upper left panel in Figure \ref{fig: pic_free_boundary_beta<1_1} corresponds to the case that two boundaries $x_\mathrm{zero}(h)$ and $x_\mathrm{aggr}(h)$ are completely separated for all $h>0$, i.e., $y_1(h)>y_2(h)$ for $h>0$ (with parameters $r=\rho=0.05$, $\mu=0.1$, $\sigma=0.25$, $\beta_1=0.2$, $\beta_2=0.3$, $k=1.5$, $\lam=0.5$);
the upper right panel in Figure \ref{fig: pic_free_boundary_beta<1_1} corresponds to the case that two boundaries $x_\mathrm{zero}(h)=x_\mathrm{aggr}(h)$ when the reference level is low that $h\leq h^*$ for some critical point $h^*>0$ (with parameters $r=\rho=0.05$, $\mu=0.1$, $\sigma=0.25$, $\beta_1=0.2$, $\beta_2=0.3$, $k=1.5$, $\lam=0.92$);
the lower left panel in Figure \ref{fig: pic_free_boundary_beta<1_1} corresponds to the case that two boundaries $x_\mathrm{zero}(h)=x_\mathrm{aggr}(h)$ when the reference level is high that $h\geq h^*$ for some $h^*>0$ (with parameters $r=\rho=0.05$, $\mu=0.1$, $\sigma=0.25$, $\beta_1=0.2$, $\beta_2=0.1$, $k=1.5$, $\lam=0.973$); and the lower right panel in Figure \ref{fig: pic_free_boundary_beta<1_1} corresponds to the case that $x_\mathrm{zero}(h)=x_\mathrm{aggr}(h)$ for all $h\geq 0$ (with parameters $r=\rho=0.05$, $\mu=0.1$, $\sigma=0.25$, $\beta_1=0.2$, $\beta_2=0.2$, $k=1.5$, $\lam=0.95$).

Second, Figure \ref{fig: pic_free_boundary_beta<1_1} illustrates again that the positive optimal consumption can never fall below the reference level, i.e., we must have $c^*(x,h)>\lambda h$ if $c^*(x,h)>0$ so that there exists a jump when the wealth process $X_t^*$ crosses the boundary curve $x_\mathrm{zero}(H_t^*)$.
In particular, for some value of $h$ such that $x_\mathrm{zero}(h) = x_\mathrm{aggr}(h)$ hold, the optimal consumption may jump from $0$ to the current maximum level $H_t^*=h$ immediately, indicating that the agent consumes at the historical maximum level $h$ if the agent starts to consume.
This differs substantially from the continuous optimal consumption process derived in Deng et al. \cite{DengLiPY2020arXiv}.
The jump of consumption is caused by the risk-loving attitude over the loss domain in the S-shaped utility, which corresponds to the linear piece of the concave envelop. In this wealth region, the agent prefers to stop the current consumption if it cannot surpass the reference level.
Therefore, our result under the S-shaped utility can depict the extreme behavior of some agents who cannot endure any positive consumption plan below the current reference.
We emphasize that, all the boundary curves in Figure \ref{fig: pic_free_boundary_beta<1_1}  are generally nonlinear functions of $h$, featuring the necessity of two dimensional state processes of $X_t$ and $H_t$ in our control problem. Only in the extreme case when $\beta_1=\beta_2$, the boundary curves can be expressed in a linear manner, and the dimension reduction can be conducted.

\begin{remark}
When risk aversion parameters satisfy $\beta_1=\beta_2$,  our problem has a homogeneous property that $\tilde{u}(x,h) = h^{\beta_1}\tilde{u}(x/h,1)$, and it is sufficient to consider the function $\hat{u}(\omega) := \tilde{u}(\omega, 1)$ to reduce the dimension. In this case, the boundary curves degenerate to boundary points for the new state variable $\omega=x/h$.
\end{remark}

Fix the model parameters $r=0.05$, $\rho=0.05$, $\mu = 0.1$, $\sigma=0.25$, $\beta_1=0.2$, $\beta_2=0.3$, $k=1.5$, and reference level $h=1$.
We numerically illustrate the sensitivity with respect to the reference degree $\lam$ with values $\lam\in\{0.1,0.3, 0.5, 0.7,0.9\}$.
The value function, the optimal feedback consumption, and the optimal feedback portfolio together with marked boundaries $x_\mathrm{zero}$, $x_\mathrm{aggr}$ and $x_\mathrm{lavs}$ are plotted in Figure \ref{fig: pic_lam}.
From all panels, one can observe that the boundary curve $x_\mathrm{zero}(1;\lam)$ is increasing in $\lam$, while boundary curves $x_\mathrm{aggr}(1;\lam)$ and $x_\mathrm{lavs}(1;\lam)$ are both decreasing in $\lam$. On one hand, one can explain that, when the agent has higher reference with the larger $\lambda$ and the current wealth is low, it is more likely that the optimal consumption falls below the reference, leading to zero consumption. Therefore, the threshold for positive consumption is getting larger for a larger $\lambda$. On the other hand, when the wealth is sufficiently large, larger reference degree $\lambda$ results in more aggressive consumption (see the middle panel) and it is more likely that the agent may lower the threshold to consume at the global maximum level even by reducing the portfolio amount. 

Moreover, when $\lam$ increases, we can also observe that $\lam H_t^*$ actually increases faster than the consumption $c_t^*$ during the life cycle, which leads to a drop of $c_t^* -\lam H_t^*$ and a decline in the value function from the left panel. From the right panel, when wealth decreases to the region $x < x_\mathrm{zero}(1;\lam)$, the optimal consumption stays at 0 due to the linear piece of the concave envelop, but the optimal portfolio is increasing in $x$ with a large slope. This can be interpreted by the fact that the agent needs to invest very aggressively to pull the wealth level back to the threshold $x_\mathrm{zero}(1;\lam)$ driven by the strong desire of positive consumption under the loss aversion preference. When wealth starts to surpass the threshold $x_\mathrm{zero}(1;\lam)$, the agent chooses the positive consumption above the reference level $\lambda h$, and the right panel shows that the agent may strategically withdraw some wealth from the risky asset account to support the high consumption plan. In addition, the higher reference degree $\lambda$ is, the more drastic decreasing in portfolio with respect to $x$ can be observed. When wealth tends to be larger, both the optimal consumption and optimal portfolio become increasing in $x$. The right panel shows that, when wealth is very large, the optimal portfolio is decreasing in the reference degree $\lam$, which is consistent with the fact that the agent needs more cash to support the more aggressive consumption as $\lam$ increases.

\begin{figure}[htbp]
\centering
\includegraphics[width=\textwidth]{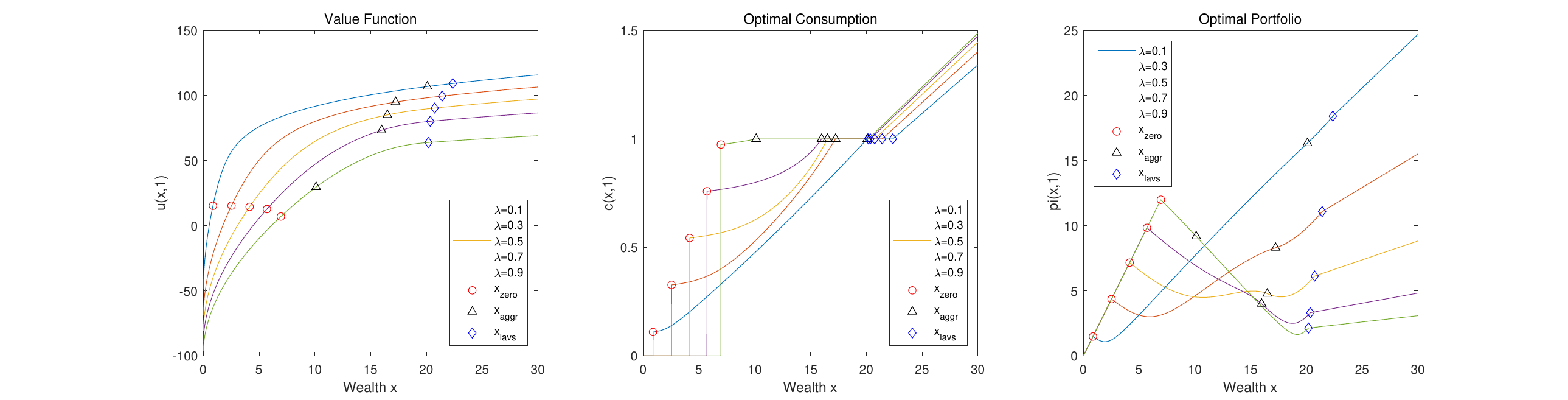}
\vspace{-0.2in}
\caption{{\small Sensitivity analysis on the reference degree $\lambda$.}}
\label{fig: pic_lam}
\end{figure}

Finally, we illustrate the sensitivity with respect to the expected return $\mu$ of the risky asset.
Fix model parameters $r=0.05$, $\rho=0.05$, $\lambda = 0.5$, $\sigma=0.25$, $\beta_1=0.2$, $\beta_2=0.3$, $k=1.5$, $h=1$ and consider $\mu\in\{0.06, 0.08, 0.1, 0.12, 0.14\}$. The value function and the optimal feedback controls together with marked boundaries are plotted in Figure \ref{fig: pic_mu}. First, the left panel shows that the value function increases in $\mu$, which matches with the intuition that the better market performance guarantees the higher wealth and the larger consumption plan. From all panels, it is interesting to observe that both boundaries $x_\mathrm{zero}(1;\mu)$ and $x_\mathrm{aggr}(1;\mu)$ are decreasing in $\mu$, while the boundary $x_\mathrm{lavs}(1;\mu)$ is increasing in $\mu$. On one hand, the higher return from the financial market secures the better wealth growth, leading to lower thresholds for the agent to start the positive consumption and the consumption at the historical peak level. On the other hand, higher return rate $\mu$ also motivates the agent to invest more in the risky asset as one can see the optimal portfolio is increasing in $\mu$ from the right panel. As a result, the agent does not blindly lower the threshold $x_\mathrm{lavs}(1;\mu)$ to create the new global consumption peak as it becomes more beneficial in the long run to invest more cash into the risky asset when $x$ is sufficiently large. Therefore, the threshold $x_\mathrm{lavs}(1;\mu)$ actually increases with $\mu$. 

One can also observe that for $x_\mathrm{zero}(1;\mu)\leq x\leq x_\mathrm{lavs}(1;\mu)$, as the expected return $\mu$ increases, the agent gradually shifts from the willingness of high consumption plan by sacrificing the portfolio to the more aggressive investment behavior to accumulate the larger wealth. Figures \ref{fig: pic_lam} and \ref{fig: pic_mu} also show that the optimal portfolio $\pi^*(x,1;\lambda,\mu)$ is decreasing in $\lambda$ but increasing in $\mu$ when $x$ is large, suggesting that some agents with the large reference degree $\lambda$  only invest more wealth in the financial market when the expected return is excessively high. This is consistent to and may partially help to explain the \textit{equity premium puzzle} (see Mehra and Prescott \cite{Mehra}) that the market premium needs to be very high to attract some agents (possibly those agents with the large reference degree if they adopt our proposed preference on consumption performance) to actively invest in the risky asset.

\begin{figure}[htbp]
\centering
\includegraphics[width=\textwidth]{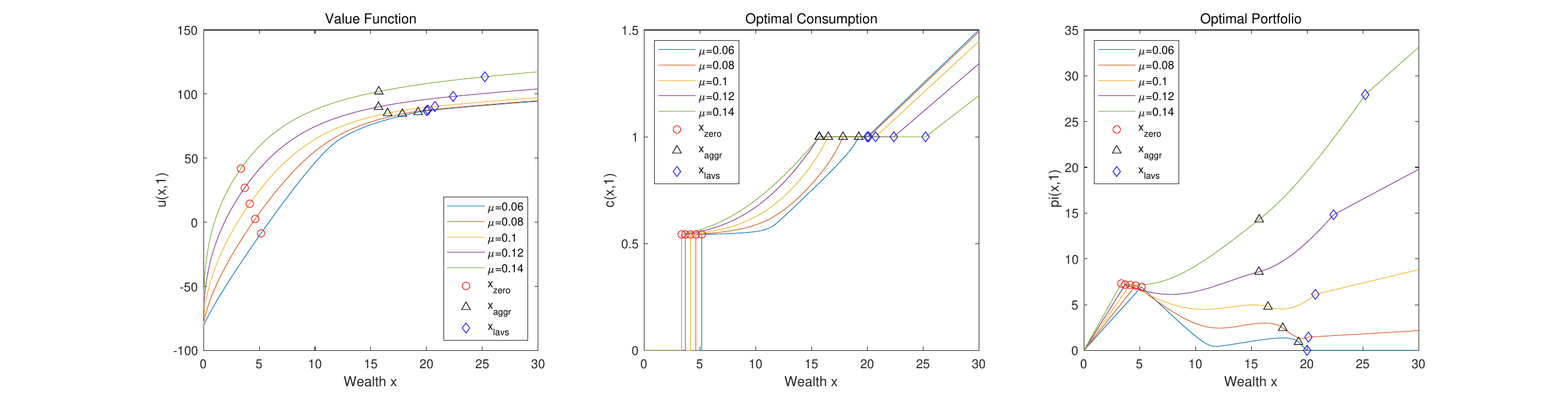}
\vspace{-0.2in}
\caption{{\small Sensitivity analysis on the expected return $\mu$.}}
\label{fig: pic_mu}
\end{figure}

\section{Proofs}\label{sec: proof}
\subsection{Proof of Proposition \ref{prop: dual_solution_beta<1}}
It is straightforward to see that the linear ODE \eqref{eq: LODE_beta<1} admits the general solution
\begin{equation*}
v(y,h) = \begin{cases}
C_1(h)y^{r_1} + C_2(h)y^{r_2} - \frac{k}{r\beta_2}(\lam h)^{\beta_2}, & \mbox{if } y> y_1(h),  \\
C_3(h)y^{r_1} + C_4(h)y^{r_2} + \frac{2}{\kappa^2\gamma_1(\gamma_1-r_1)(\gamma_1-r_2)}y^{\gamma_1} - \frac{\lam h}{r}y, & \mbox{if } y_2(h) \leq y \leq y_1(h), \\
C_5(h)y^{r_1} + C_6(h)y^{r_2} + \frac{1}{r\beta_1}((1-\lam)h)^{\beta_1} - \frac{h}{r}y, & \mbox{if } y_3(h) \leq y < y_2(h),
\end{cases}
\end{equation*}
where $C_1(\cdot), \cdots, C_6(\cdot)$ are functions of $h$ to be determined.

The free boundary condition $v_y(y,h)\rightarrow0$ in \eqref{eq: boundary_dual_y->infty} implies that $y\rightarrow+\infty$.
Together with free boundary conditions in \eqref{eq: boundary_dual_y->infty} and the formula of $v(y,h)$ in the region $y\geq y_1(h)$, we deduce that $C_1(h)\equiv 0$. To determine the remaining parameters, we consider the smooth-fit conditions with respect to the variable $y$ along $y = y_1(h)$ and $y = y_2(h)$ that
\begin{equation}\label{eq: dual_smooth_beta<1}
\begin{aligned}
&-C_3(h)y_1(h)^{r_1} + (C_2(h)-C_4(h))y_1(h)^{r_2} \\
&= \frac{k}{r\beta_2}(\lam h)^{\beta_2} + \frac{2}{\kappa^2\gamma_1(\gamma_1-r_1)(\gamma_1-r_2)}y_1(h)^{\gamma_1} - \frac{\lam h}{r}y_1(h), \\
&-r_1C_3(h)y_1(h)^{r_1-1} + r_2(C_2(h)-C_4(h))y_1(h)^{r_2-1} \\
&= \frac{2}{\kappa^2(\gamma_1-r_1)(\gamma_1-r_2)}y_1(h)^{\gamma_1-1} - \frac{\lam h}{r}, \\
&(C_3(h)-C_5(h))y_2(h)^{r_1} + (C_4(h)-C_6(h))y_2(h)^{r_2}  \\
&= \frac1{r\beta_1}((1-\lam)h)^{\beta_1}- \frac{2}{\kappa^2\gamma_1(\gamma_1-r_1)(\gamma_1-r_2)}y_2(h)^{\gamma_1}-\frac{(1-\lam)h}{r}y_2(h) , \\
& r_1(C_3(h)-C_5(h))y_2(h)^{r_1-1} + r_2(C_4(h)-C_6(h))y_2(h)^{r_2-1} \\
&= -\frac{2}{\kappa^2(\gamma_1-r_1)(\gamma_1-r_2)}y_2(h)^{\gamma_1-1}
-\frac{(1-\lam)h}{r}.
\end{aligned}
\end{equation}
The equations in \eqref{eq: dual_smooth_beta<1} can be treated as linear equations for $C_3(h)$, $C_2(h) - C_4(h)$, and $C_3(h)-C_5(h)$ and $C_4(h) - C_6(h)$.
By solving the system of equations, we can obtain
\begin{equation*}
\begin{aligned}
C_3(h) &= \frac{y_1(h)^{-r_1}}{r(r_1-r_2)}\bigg(\frac{kr_2}{\beta_2}(\lam h)^{\beta_2} + \frac{r_1r_2}{\gamma_1(\gamma_1-r_1)}y_1(h)^{\gamma_1} + \lam h r_1 y_1(h) \bigg),\\
C_2(h) - C_4(h) &= \frac{y_1(h)^{-r_2}}{r(r_1-r_2)}\bigg(\frac{kr_1}{\beta_2}(\lam h)^{\beta_2} + \frac{r_1r_2}{\gamma_1(\gamma_1-r_2)}y_1(h)^{\gamma_1} + \lam h r_2 y_1(h) \bigg),\\
C_3(h) - C_5(h) &= \frac{y_2(h)^{-r_1}}{r(r_1-r_2)}\bigg( -\frac{r_2}{\beta_1}((1-\lam)h)^{\beta_1}+ \frac{r_1r_2}{\gamma_1(\gamma_1-r_1)}y_2(h)^{\gamma_1} - (1-\lam) h r_1y_2(h) \bigg),\\
C_4(h) - C_6(h) &= \frac{y_2(h)^{-r_2}}{r(r_1-r_2)}\bigg( \frac{r_1}{\beta_1}((1-\lam)h)^{\beta_1} - \frac{r_1r_2}{\gamma_1(\gamma_1-r_2)}y_2(h)^{\gamma_1} + (1-\lam) h r_2y_2(h) \bigg).
\end{aligned}
\end{equation*}
Therefore, $C_2(h)$ to $C_5(h)$ can be expressed by \eqref{paraC}. To solve $C_2(h),C_4(h)$ and $C_6(h)$, we find $C_6(h)$ first, and $C_4(h)$ and $C_2(h)$ can then be determined.

Indeed, as $h\rightarrow+\infty$, we get $y\rightarrow0$ in the region $y_3(h) \leq y < y_2(h)$, and the boundary condition \eqref{eq: boundary_dual_y->0_2} leads to
$\lim\limits_{h\rightarrow+\infty }
\frac{h}{v_y(y_3(h),h)} = C$, where $C$ is a positive constant.
Along the free boundary, we have
$v_y(y_3(h),h) = r_1C_5(h)y_3(h)^{r_1-1} + r_2C_6(h)y_3(h)^{r_2-1} + \frac{h}{r}$. It follows from $\lim\limits_{h\rightarrow+\infty }
\frac{h}{v_y(y_3(h),h)} > 0$ that $v_y(y_3(h), h) = O(h)$ as $h\rightarrow+\infty$.
Therefore, we deduce that
$C_6(h) = O(C_5(h)h^{(r_1-r_2)(\beta_1-1)}) + O(h^{r_1(\beta_1-1) +1 })$.
By Lemma \ref{remark: C_order} and the definition $y_3(h) = (1-\lam)^{\beta_1}h^{\beta_1-1}$, it follows that
$$\begin{aligned}
C_6(h) &= O(C_5(h)h^{(r_1-r_2)(\beta_1-1)}) + O(h^{r_1(\beta_1-1)+1})  \\
&= O(h^{(r_1-r_2)(\beta_1-1) + r_2\beta_1+r_1+(\beta_2-\beta_1) }) + O(h^{(r_1-r_2)(\beta_1-1) + r_2\beta_1+r_1}) + O(h^{r_1\beta_1+r_2}) \\
&= O(h^{r_1\beta_2+r_2}) + O(h^{r_1\beta_1+r_2}),
\end{aligned}$$
as $h\rightarrow+\infty$, where the last equation holds because
$$\min(r_1\beta_1+r_2, r_1\beta_2+r_2) \leq r_1\beta_1+r_2 + (\beta_2-\beta_1) \leq \max(r_1\beta_1+r_2, r_1\beta_2+r_2).$$
From \textbf{Assumption (A1)}, it follows that $\lim\limits_{h\rightarrow+\infty} C_6(h) = 0$.
Therefore, we can write $C_6(h) = -\int_h^\infty C_6'(s)ds$. We then apply the free boundary condition \eqref{eq: free_boundary_dual} at $y_3(h) = (1-\lam)^{\beta_1}h^{\beta_1-1}$ that
$$
C'_5(h)y_3(h)^{r_1} + C'_6(h)y_3(h)^{r_2} + \frac{1}{r}(1-\lam)^{\beta_1}h^{\beta_1-1} - \frac{1}{r}y_3(h) = 0,
$$
which implies the desired result of $C_6(h)$ in \eqref{paraC}.

\subsection{Proof of Theorem  \ref{thm: beta<1} (Verification Theorem)}\label{sec:vefproof}
The proof of the verification theorem boils down to show that the solution of the auxiliary HJB variational inequality \eqref{eq: HJB_eqn} coincides with the value function, i.e. there exists $(\pi^*, c^*)\in\mathcal{A}(x)$ such that $\tu(x,h) = \mathbb{E}[ \int_0^\infty e^{-rt}\tU(c_t^*, H_t^*)dt ]$.
For any admissible strategy $(\pi, c) \in \mathcal{A}(x)$, we have $\mathbb{E}[ \int_0^\infty c_tM_tdt] \leq x$ by the supermartingale property and the standard budget constraint argument, see Karatzas et al. \cite{KLSX}.

Let $(\lam, h)$ be regarded as parameters, the dual transform of $U$ with respect to $c$ in the constrained domain that $V(q,h) := \sup_{c\geq0} [\tU(c, h)-cq]$ defined in \eqref{eq: LODE_beta<1}. Moreover, $V$ can be attained by the construction of the feedback optimal control $c^\dag(y,h)$ in \eqref{eq: c_y_beta<1}.

In what follows, we distinguish the two reference processes, namely $H_t := h\vee \sup_{s\leq t}c_s$ and $H_t^\dag(y):= h\vee \sup_{s\leq t} c^\dag(Y_s(y), H_s^\dag(y))$ that correspond to the reference process under an arbitrary consumption process $c_t$ and under the optimal consumption process $c^\dag$ with an arbitrary $y>0$.
Note that the global optimal reference process is defined later by $H_t^* := H_t^\dag(y^*)$ with $y^*>0$ to be determined.
Let us now further introduce
\begin{equation}\label{eq: H_hat}
\hat{H}_t(y) := h \vee
\bigg((1-\lam)^{-\frac{\beta_1}{\beta_1-1}} (\inf_{s\leq t}Y_s(y))^{\frac{1}{\beta_1-1}} \bigg),
\end{equation}
where $Y_t(y) = ye^{rt}M_t$ is the discounted martingale measure density process.

For any admissible controls $(\pi, c) \in \mathcal{A}(x)$, recall the reference process $H_t = h\vee \sup_{s\leq t}c_s$, and for all $y>0$, we see that
\begin{equation}\label{eq: verification_main}
\begin{aligned}
\mathbb{E}\bigg[ \int_0^\infty e^{-rt}\tU(c_t, H_t)dt \bigg]
&= \mathbb{E}\bigg[ \int_0^\infty e^{-rt}(\tU(c_t, H_t)-Y_t(y)c_t)dt \bigg]
+ y\mathbb{E}\bigg[ \int_0^\infty c_tM_tdt \bigg] \\
&\leq \mathbb{E}\bigg[ \int_0^\infty e^{-rt} V(Y_t(y), H_t^\dag(y))dt \bigg] + yx \\
&= \mathbb{E}\bigg[ \int_0^\infty e^{-rt} V(Y_t(y), \hat{H}_t(y))dt \bigg] + yx \\
&= v(y,h) + yx,
\end{aligned}
\end{equation}
the third equation holds because of Lemma \ref{lemma: 2}, and the last equation is verified by Lemma \ref{lemma: 1}.
In addition, Lemma \ref{lemma: 3} guarantees the equality with the choice of $c_t^* = c^\dag(Y_t(y^*), H_t^\dag(y^*))$, in which $y^*$ satisfies that $\mathbb{E}\big[ \int_0^\infty c^\dag(Y_t(y^*), H_t^\dag(y^*))M_tdt \big] = x$ for a given $x\in\mathbb{R}_+$ and $h\geq0$. In conclusion, we have that
\begin{align*}
\sup\limits_{(\pi, c)\in\mathcal{A}(x)}
\mathbb{E}\bigg[ \int_0^\infty e^{-rt}\tU(c_t, H_t) dt\bigg]
= \inf_{y>0}(v(y,h) + yx) = \tu(x,h).
\end{align*}

Then we prove some auxiliary results that have been used above.
We first need some asymptotic results on the coefficients in Proposition \ref{prop: dual_solution_beta<1}.
\begin{lemma}\label{remark: C_order}
Based on the semi-analytical forms in Proposition \eqref{prop: dual_solution_beta<1}, we have that
$$\begin{aligned}
C_2(h) &= O\big( h^{\beta_2}w(h)^{-r_2(\beta_1-1)} \big) + O\big(w(h)^{r_1\beta_1+r_2}\big) + O\big(w(h)^{(\gamma_2-r_2)(\beta_1-1)}\big), \\
C_3(h) &= O\big(w(h)^{r_2\beta_1+r_1+(\beta_2-\beta_1)}\big) + O\big(w(h)^{r_2\beta_1+r_1}), \\
C_4(h) &= O\big(h^{r_1\beta_1+r_2 + (\beta_2-\beta_1)}\big) +O \big(h^{r_1\beta_1+r_2}\big), \\
C_5(h) &= O\big(h^{r_2\beta_1+r_1+(\beta_2-\beta_1)}\big) + O\big(h^{r_2\beta_1+r_1}\big), \\
C_6(h) &= O\big(h^{r_1\beta_1+r_2 + (\beta_2-\beta_1)}\big) +O \big(h^{r_1\beta_1+r_2}\big),
\end{aligned}$$
where $\gamma_2 = \frac{\beta_2}{\beta_2-1}$, and $w(h)$ is defined in \eqref{eq: def_w}.
Moreover, the function $w(h)$ satisfies
$$
w(h) = O(h), ~~w(h)^{-1} = O(h^{-1}) + O(h^{-\frac{\beta_2-1}{\beta_1-1}}), ~~h = O(w(h)) + O(w(h)^{\frac{\beta_2-1}{\beta_1-1}}),~~ h^{-1} = O(w(h)^{-1}).
$$
\end{lemma}

\begin{proof}
We first discuss the asymptotic results of $y_1(h)$ and $y_2(h)$.
It is easy to see that $(1-\lam)^{\beta_1}h^{\beta_1-1} = y_3(h) < y_2(h) \leq ((1-\lam)h)^{\beta_1-1} $ and thus $y_2(h) = O(h^{\beta_1-1})$, $y_2(h)^{-1} = O(h^{1-\beta_1})$.
Moreover, if $y_1(h) > y_2(h)$, we have $y_1(h) = w(h)^{\beta_1-1}$; if $y_1(h) = y_2(h)$, indicating that $w(h) = (1-\lam)h$, thus we have $ y_1(h) \leq ((1-\lam)h)^{\beta_1-1} = w(h)^{\beta_1-1}$ and $y_1(h) > y_3(h) = (1-\lam)^{\beta_1}h^{\beta_1-1} = (1-\lam)w(h)^{\beta_1-1}$.
Therefore, we have $y_1(h) = O(w(h)^{\beta_1-1})$ and $y_1(h)^{-1} = O(w(h)^{1-\beta_1})$.

To obtain the asymptotic properties of $C_2(h)$ to $C_6(h)$, we need to derive the asymptotic property of $w(h)$.
If $y_1(h) > y_2(h)$, the equation \eqref{eq: beta<1,tagent z} indicates that
\begin{equation}\label{eq: tagent z_eq}
\frac{1-\beta_1}{\beta_1} \bigg(\frac{w(h)}{h}\bigg)^{\beta_1} + \frac{k}{\beta_2}\lam^{\beta_2}h^{\beta_2-\beta_1} - \lam \bigg(\frac{w(h)}h\bigg)^{\beta_1-1} = 0,
\end{equation}
where $0<\beta_1<1$, $0<\beta_2<1$, $0<\lam<1$ and $h>0$.
We obtain the asymptotic property of $w$ in two cases: $\beta_1<\beta_2$ and $\beta_1>\beta_2$ respectively.
In the sequel of the proof below, let $C>0$ be a generic positive constant independent of $(x,h)$, which may be different from line to line.
If $\beta_1 < \beta_2$, as $h\rightarrow+\infty$,
the second term of equation \eqref{eq: tagent z_eq} goes to infinity,
yielding $(\frac{w(h)}{h})^{\beta_1-1} - Ch^{\beta_2-\beta_1} \rightarrow0$ and thus $w \geq C h^{\frac{\beta_2-1}{\beta_1-1}}$;
as $h\rightarrow0$, the second term goes to 0,
yielding $(\frac{w(h)}{h})^{\beta_1} - C(\frac{w(h)}{h})^{\beta_1-1}\rightarrow0$ and thus $w(h)\geq Ch$.
If $\beta_1 > \beta_2$, we can similarly obtain that $w(h) \geq Ch$ and $w(h)\geq Ch^{\frac{\beta_2-1}{\beta_1-1}}$ as $h$ goes to infinity and 0 respectively.
Together with the fact that $w(h) \leq (1-\lam)h$, we deduce that
\begin{equation*}
h = O(w(h)) + O(w(h)^{\frac{\beta_1-1}{\beta_2-1}}) , ~\mathrm{and}~
            h^{-1} = O(w(h)^{-1}) .
\end{equation*}

If $y_1(h) = w(h)^{\beta_1-1} > y_2(h) = ((1-\lam)h)^{\beta_1-1}$, then $y_1'(h) = (\beta_1-1)w(h)^{\beta_1-2}w'(h) = O\big(w(h)^{\beta_1-2}w'(h)\big)$, and $y_2'(h) = O\big( h^{\beta_1-2} \big)$.
If $y_1(h) = y_2(h) = \frac{k}{\beta_2}\lam^{\beta_2}h^{\beta_2-1} + \frac{1}{\beta_1}(1-\lam)^{\beta_1}h^{\beta_1-1} $,
then $w(h) = (1-\lam)h$, $w'(h) = 1-\lam$, and thus $y_1'(h) = \frac{k}{\gamma_2}\lam^{\beta_2}h^{\beta_2-2} + \frac{1}{\gamma_1}(1-\lam)^{\beta_1}h^{\beta_1-1} = O\big( h^{-1}y_1(h) \big) = O(h^{-1}w(h)^{\beta_1-1}) = O(w(h)^{\beta_1-2}w'(h)),$
and $y_2'(h) = O\big( \frac{y_2(h)}{h} \big) = O\big( h^{\beta_1-2} \big)$.
In summary, we have $y_1'(h) = O(w(h)^{\beta_1-2}w'(h))$ and $y_2'(h) = O\big( h^{\beta_1-2} \big)$.

We further discuss the asymptotic property of $w'(h)$.
If $w(h) = (1-\lam)h$, it is obvious that $w'(h) = 1-\lam = O(1)$.
Otherwise, 
we have
$$
w'(h) = \frac{\lam }{1-\beta_1}\cdot \frac{w(h)^{\beta_1-1} - k(\lam h)^{\beta_2-1}}{w(h)^{\beta_1-1} + \lam h w(h)^{\beta_1-2}}.
$$
Since $\lam h w(h)^{\beta_1-1} > \frac{k}{\beta_2}(\lam h)^{\beta_2} > k(\lam h)^{\beta_2}$, we can derive that $w'(h) > 0$, $w'(h) < C$, $w'(h) = O(1)$, and $hw'(h) = O(w(h))$.

Based on the asymptotic property of $y_1(h)$ and $y_2(h)$, we can find the asymptotic results of $C_2(h)$ to $C_6(h)$. Let us begin with $C_3(h)$ and $C_5(h)$.
It is easy to see that
$$
C_3(h) = O(w(h)^{r_2\beta_1+r_1+(\beta_2-\beta_1)}) + O(w(h)^{r_2\beta_1+r_1}).
$$
Note that
$$\begin{aligned}
C_3(h) &= \frac{1}{r(r_1-r_2)}\bigg\{\frac{kr_2}{\beta_2}(\lam h)^{\beta_2}y_1(h)^{-r_1} + \frac{r_1r_2}{(\gamma_1-r_1)\gamma_1}y_1(h)^{\gamma_1-r_1} + \lam h r_1 y_1(h)^{r_2} \bigg\} \\
&= C^1h^{\beta_2}y_1(h)^{-r_1} + C^2y_1(h)^{\gamma_1-r_1} + C^3hy_1(h)^{r_2},
\end{aligned}$$
and
$$\begin{aligned}
C_3(h) - C_5(h) &= \frac{y_2(h)^{-r_1}}{r(r_1-r_2)}\bigg( -\frac{r_2}{\beta_1}((1-\lam)h)^{\beta_1}+ \frac{r_1r_2}{\gamma_1(\gamma_1-r_1)}y_2(h)^{\gamma_1} - (1-\lam) h r_1y_2(h) \bigg)\\
&= C^1h^{\beta_1}y_2(h)^{-r_1} + C^2y_2(h)^{\gamma_1-r_1} + C^3 hy_2(h)^{r_2},
\end{aligned}$$
where $C^1$ to $C^3$ are discriminant constants in each equation.
Then by $y_1(h) = O\big( w(h)^{\beta_1-1} \big)$, $y_1(h)^{-1} = O\big( w(h)^{1-\beta_1} \big) = O\big( h^{1-\beta_1} \big)$, $y_1'(h) = O\big( w(h)^{\beta_1-2}w'(h) \big)$, $y_2(h) = O\big( h^{\beta_1-1} \big)$, $y_2(h)^{-1} = O\big( h^{1-\beta_1} \big)$, $y_2'(h) = O\big( h^{\beta_1-2} \big)$,
$w(h) = O(h)$, $w'(h) = O(1)$ and $hw'(h) = O(w(h))$, we have
$$\begin{aligned}
C_3'(h) =& C^1 h^{\beta_2-1}y_1(h)^{-r_1} + C^2 h^{\beta_2}y_1(h)^{-r_1-1}y_1'(h)
+ C^3 y_1(h)^{\gamma_1-r_1-1}y_1'(h) \\
&+ C^4 y_1(h)^{r_2} + hy_1(h)^{r_2-1}y_1'(h) \\
=& O\big( h^{r_2(\beta_1-1)+(\beta_2-\beta_1)} \big) + O\big( h^{r_2(\beta_1-1)} \big),
\end{aligned}$$
and
$$\begin{aligned}
C_3'(h) - C_5'(h) =& C^1h^{\beta_1-1}y_2(h)^{-r_1} + C^2 h^{\beta_1}y_2(h^{-r_1-1}y_2'(h)) + C^3y_2(h)^{\gamma_1-r_1-1}y_2'(h) \\
&+ C^4y_2(h)^{r_2} + C^5 hy_2(h)^{-r_2-1}y_2'(h) \\
=& O(h^{r_2(\beta_1-1)+(\beta_2-\beta_1)}) + O(h^{r_2(\beta_1-1)}) ,
\end{aligned}$$
where $C^1$ to $C^5$ are discriminant constants, and thus
$$
C_5'(h) = O(h^{r_2(\beta_1-1)+(\beta_2-\beta_1)}) + O(h^{r_2(\beta_1-1)}).
$$

Recall that
$$\begin{aligned}
C'_6(h) = -(1-\lam)^{(r_1-r_2)\beta_1} C'_5(h) h^{(r_1-r_2)(\beta_1-1)}= O(h^{r_1(\beta_1-1) + (\beta_2-\beta_1)})
+O(h^{r_1(\beta_1-1)}).
\end{aligned}$$
We can get the asymptotic property of $C_6(h)$ that
$$\begin{aligned}
C_6(h) &= -\int_h^\infty C_6'(h)dh = O(h^{r_1\beta_1+r_2 + (\beta_2-\beta_1)})
+O(h^{r_1\beta_1+r_2}) .
\end{aligned}$$ Finally, it follows that
$$
C_4(h) = O(h^{r_1\beta_1+r_2 + (\beta_2-\beta_1)}) +O (h^{r_1\beta_1+r_2}).
$$
and
$$C_2(h) = O(h^{\beta_2}w(h)^{-r_2(\beta_1-1)}) + O(w(h)^{r_1\beta_1+r_2}) + O(w(h)^{(\gamma_2-r_2)(\beta_1-1)}), $$
in view that $h = O(w(h)) + O\big(w(h)^{\frac{\beta_1-1}{\beta_2-1}}\big)$.
\end{proof}

Following similar proofs of Lemma 5.1 and Lemma 5.2 in Deng et al. \cite{DengLiPY2020arXiv} and using asymptotic results in Lemma \ref{remark: C_order}, we can readily obtain the next two lemmas.
\begin{lemma}\label{lemma: 1}
For any $y>0$ and $h\geq0$, the dual transform $v(y,h)$ of value function $\tu(x,h)$ satisfies
$$
v(y,h) = \mathbb{E}\bigg[ \int_0^\infty e^{-rt}V(Y_t(y), \hat{H}_t(y)) dt\bigg],
$$
where $V(\cdot,\cdot)$ is defined in \eqref{eq: LODE_beta<1}, $Y_t(\cdot)$ and $\hat{H}_t(\cdot)$ are defined in \eqref{eq: H_hat}.
\end{lemma}

\begin{lemma}\label{lemma: 2}
Let $V(\cdot,\cdot)$, $Y_t$, $H_t^*$ and $\hat{H}_t$ be the same as in Lemma \ref{lemma: 1}, then for all $y>0$, we have $H_t^\dag = \hat{H}_t(y)$, $t\geq0$, and hence
$$
\mathbb{E}\bigg[ \int_0^\infty e^{-rt} V(Y_t(y), H_t^\dag(y))dt \bigg]
= \mathbb{E}\bigg[ \int_0^\infty e^{-rt} V(Y_t(y), \hat{H}_t(y))dt \bigg].
$$
\end{lemma}
%
Let us then continue to prove some other auxiliary results.

\begin{lemma}\label{lemma: 3}
The inequality in \eqref{eq: verification_main} becomes equality with $c_t^* = c^\dag(Y_t(y^*), \hat{H}_t(y^*))$, $t\geq0$, with $y^* = y^*(x,h)$ as the unique solution to
\begin{equation}\label{eq: ECM<=x}
\mathbb{E}\bigg[\int_0^\infty c^\dag(Y_t(y^*), \hat{H}_t(y^*))M_tdt\bigg] = x.
\end{equation}
\end{lemma}
\begin{proof}
By the definition of $V$, it is obvious that for all $(\pi, c) \in \mathcal{A}(x)$, $\tU(c_t, H_t) - Y_t(y)c_t \leq V(Y_t(y), H_t)$.
Moreover, the inequality becomes an equality with the optimal feedback $c^\dag(Y_t(y), H_t^\dag(y))$.
Thus, it follows that
$$
\int_0^\infty e^{-rt}(\tU(c_t, H_t) - Y_t(y)c_t)dt \leq \int_0^\infty e^{-rt}V(Y_t(y), H_t^\dag(y))dt.
$$
To turn \eqref{eq: verification_main} into an equality, the equality of \eqref{eq: ECM<=x} needs to hold with some $y^*>0$ to be determined later, and
\begin{equation}\label{eq: UV_eq}
\tU(c_t, H_t) - Y_t(y) c_t = V(Y_t(y), H_t)
\end{equation}
also needs to hold.
Hence, we choose to employ
$c^\dag(Y_t(y), \hat{H}_t(y)) := \hat{H}_t(y)F_t(y, Y_t(y))$,
where
$$\begin{aligned}
F_t(y,z) &:= \mathbb{I}_{\{y_3(\hat{H}_t(y)) \leq z < y_2(\hat{H}_t(y))\}} + \bigg(\lam + \frac{z^{\frac1{\beta_1-1}}}{\hat{H}_t(y)}\bigg)\mathbb{I}_{\{y_2(\hat{H}_t(y)) \leq z < y_1(\hat{H}_t(y))\}}.
\end{aligned}$$

It follows from definition that:
(i) If $y\rightarrow0_+$, then $\hat{H}_t(y) \rightarrow +\infty$ and $F_t(y, Y_t(y)) > 0$, it indicates that $\mathbb{E}[\int_0^\infty M_tc^\dag(Y_t(y), \hat{H}_t(y))dt] \rightarrow+\infty$;
(ii) If $y\rightarrow+\infty$, then $\hat{H}_t(y) \rightarrow h_+$ and $F_t(y, Y_t(y)) \rightarrow 0_+$, it indicates that $\mathbb{E}[\int_0^\infty M_tc^\dag(Y_t(y), \hat{H}_t(y))dt] \rightarrow 0_+$.
The existence of $y^*$ can thus be verified if $\mathbb{E}[\int_0^\infty M_tc^\dag(Y_t(y), \hat{H}_t(y))dt]$ is continuous in $y$.

Indeed, let $c^\ddag(Y_t(y), \hat{H}_t(y)) = \max(c^\dag(Y_t(y), \lam\hat{H}_t(y)))$, then
$\mathbb{E}[\int_0^\infty M_tc^\ddag(Y_t(y), \hat{H}_t(y))dt]$ exists and is continuous in $y$, and
$$\mathbb{E}[\int_0^\infty M_tc^\dag(Y_t(y), \hat{H}_t(y))dt] = \mathbb{E}[\int_0^\infty M_tc^\ddag(Y_t(y), \hat{H}_t(y))\mathbf{1}\{Y_t(y) \leq y_1(\hat{H}_t(y))\} dt].$$
Therefore, $\mathbb{E}[\int_0^\infty M_tc^\dag(Y_t(y), \hat{H}_t(y))dt]$ is also continuous in $y$.
\end{proof}

\begin{lemma}\label{lemma: 4}
The following transversality condition holds that for all $y>0$,
$$
\lim\limits_{T\rightarrow+\infty} \mathbb{E}\bigg[ e^{-rT}v(Y_T(y), \hat{H}_T(y)) \bigg] = 0.
$$
\end{lemma}

\begin{proof}
Recall that $
\hat{H}_t(y) := h \vee
\bigg((1-\lam)^{-\frac{\beta_1}{\beta_1-1}} (\inf_{s\leq t}Y_s(y))^{\frac{1}{\beta_1-1}} \bigg)$.
In this proof, the results in Lemma \ref{lemma: claim1} and Lemma \ref{lemma: claim2} are applied repeatedly, therefore, we omit the illustrations if there is no ambiguity.
In more details, we use Lemma \ref{lemma: claim1} with $\beta = \beta_2 \geq \min(\beta_1,\beta_2)$, and use Lemma \ref{lemma: claim2} with $\gamma = \gamma_1, \gamma_2,$ and $\frac{\beta_2}{\beta_1-1}$ since $r_1 > 0 > \frac{\beta_2}{\beta_1-1} \geq \min(\gamma_1, \gamma_2) > r_2$, which can be obtained by some simple computations.

Let us firstly consider the case $c_T = 0$.
We first write that
\begin{equation}\label{eq: trans_R1}
e^{-rT}\E[v(Y_T(y),\hat{H}_T(y))] = e^{-rT}\E\bigg[C_2(\hat{H}_T(y)) Y_T(y)^{r_2} - \frac{k}{r\beta_2}(\lam \hat{H}_T(y))^{\beta_2} \bigg],
\end{equation}
in which the second term converges to 0 as $T\rightarrow+\infty$ due to Lemma \ref{lemma: claim1}.
For the first term in \eqref{eq: trans_R1}, since $Y_T(y) > y_1(\hat{H}_T(y)) \geq w_T(y)^{\beta_1-1} $, we have
$$\begin{aligned}
e^{-rT}\E\bigg[ C_2(\hat{H}_T(y)) (Y_T(y))^{r_2} \bigg]
&= e^{-rT}O(\E[\hat{H}_T^{\beta_2}(y)w_T^{-r_2(\beta_1-1)}Y_T(y)^{r_2}]) + e^{-rT}O(\E[w_T(y)^{r_1\beta_1+r_2}Y_T(y)^{r_2}]) \\
&~~+e^{-rT}O(\E[w_T^{(\gamma_2-r_2)(\beta_1-1)}Y_T(y)^{r_2}]) \\
&= e^{-rT}O(\E[\hat{H}_T^{\beta_2}(y)]) + e^{-rT}O(\E[Y_T(y)^{\gamma_1}])+ e^{-rT}O(\E[Y_T(y)^{\gamma_2}]),
\end{aligned}$$
which vanishes as $T\rightarrow+\infty$ due to Lemma \ref{lemma: claim1} and Lemma \ref{lemma: claim2}.

We then consider the case $0<c_T<\hat{H}_T(y)$.
In this case, $y_2(\hat{H}_T(y)) \leq Y_T(y) \leq y_1(\hat{H}_T(y))$, and thus
\begin{equation}\label{eq: trans_R2}\begin{aligned}
e^{-rT}v(Y_T(y),\hat{H}_T(y)) = e^{-rT}&\bigg[
C_3(\hat{H}_T(y)) Y_T(y)^{r_1} + C_4(\hat{H}_T(y)) Y_T(y)^{r_2}\\
&+ \frac{2}{\kappa^2\gamma_1(\gamma_1-r_1)(\gamma_1-r_2)}Y_T(y)^{\gamma_1} - \frac{\lam \hat{H}_T(y)}{r}Y_T(y) \bigg].
\end{aligned}\end{equation}
We consider asymptotic behavior of the above equation term by term as $T\rightarrow+\infty$.

The third term in \eqref{eq: trans_R2} clearly converges to 0 by Lemma \ref{lemma: claim1}. For the fourth term in \eqref{eq: trans_R2}, since $Y_T(y) \leq y_1(\hat{H}_T(y)) = O(\hat{H}_T(y)^{\beta_2-1})+O(\hat{H}_T(y)^{\beta_1-1})$, we have
$$\begin{aligned}
\E[e^{-rT}Y_T(y)\hat{H}_T(y)]
&= e^{-rT}O(\E[\hat{H}_T(y)^{\beta_2}]) + e^{-rT}O(\E[\hat{H}_T(y)^{\beta_1}]),
\end{aligned}$$
which also vanishes as $T\rightarrow+\infty$ by Lemma \ref{lemma: claim1}.

Let us continue to consider the terms containing $C_3(\hat{H}_T(y))$ and $C_4(\hat{H}_T(y))$ in equation \eqref{eq: trans_R2}.
Because of the constraint
$w_T(y) = O(Y_t(y)^{\frac1{\beta_1-1}})$ due to $Y_t(y) \leq y_1(\hat{H}_T(y)) \leq w_T(y)^{\beta_1-1}$ which is discussed in the proof of Remark \eqref{remark: C_order},  we deduce that
$$\begin{aligned}
&e^{-rT}\E\bigg[ C_3(\hat{H}_T(y)) (Y_T(y))^{r_1} \bigg]\\
=& e^{-rT}O(\E[w_T(y)^{r_1+r_2\beta_1+(\beta_2-\beta_1)}Y_T(y)^{r_1}]) + e^{-rT}O(\E[w_T(y)^{r_1+r_2\beta_1}(Y_T(y))^{r_1}]) \\
=& e^{-rT}O(\E[Y_T(y)^{\frac{\beta_2}{\beta_1-1}}]) + e^{-rT}O(\E[Y_T(y)^{\gamma_1}]),
\end{aligned}$$
which converges to 0 by Lemma \ref{lemma: claim2}. \\

In addition, from $Y_T(y) \geq (1-\lam)^{\beta_1} \hat{H}_T(y)^{\beta_1-1}$, it follows that $\hat{H}_T(y)^{-1} = O(Y_T(y)^{\frac1{1-\beta_1}})$, and thus
$$\begin{aligned}
&e^{-rT}\E\bigg[ C_4(\hat{H}_T(y)) (Y_T(y))^{r_2} \bigg]\\
=& e^{-rT}O(\E[\hat{H}_T(y)^{r_1\beta_1+r_2+(\beta_2-\beta_1)}Y_T(y)^{r_2}]) + e^{-rT}O(\E[\hat{H}_T(y)^{r_1\beta_1+r_2}(Y_T(y))^{r_2}]) \\
=& e^{-rT}O(\E[Y_T(y)^{\frac{\beta_2}{\beta_1-1}}]) + e^{-rT}O(\E[Y_T(y)^{\gamma_1}]),
\end{aligned}$$
which vanishes as $T\rightarrow+\infty$ by Lemma \ref{lemma: claim2}.

Finally, we consider the case $c_T = \hat{H}_T(y)$ and write that
\begin{equation}\label{eq: trans_R3}\begin{aligned}
e^{-rT}v(Y_T(y), \hat{H}_T(y)) &= e^{-rT}\bigg( C_5(\hat{H}_T(y))Y_T(y)^{r_1} + C_6(\hat{H}_T(y))Y_T(y)^{r_2} \\
&~~+ \frac{1}{r\beta_1}((1-\lam)\hat{H}_T(y))^{\beta_1} - \frac{1}{r}\hat{H}_T(y)Y_T(y)\bigg).
\end{aligned}\end{equation}
In this case, similar to the discussion for \eqref{eq: trans_R2}, we have $w_T(y) = O(Y_T(y)^{\frac1{\beta_1-1}}).$

The last two terms of right-hand side in equation \eqref{eq: trans_R3}, similar to the last two terms of right-hand side in equation \eqref{eq: trans_R2}, also converge to 0 as $T\rightarrow+\infty$.

For the first term in \eqref{eq: trans_R3}, by Remark \ref{remark: C_order}, we have
$$\begin{aligned}
e^{-rT}C_5(\hat{H}_T(y))Y_T(y)^{r_1}
&= e^{-rT}\bigg( O(w_T(y)^{r_2\beta_1+r_1+(\beta_2-\beta_1)}) + O(w_T(y)^{r_2\beta_1+r_1})  \bigg)Y_T(y)^{r_1} \\
&= e^{-rT}O\bigg( Y_T(y)^{\frac{\beta_2}{\beta_1-1}} \bigg) + e^{-rT}O\bigg( Y_T(y)^{\gamma_1}\bigg),
\end{aligned}$$
which converges to 0 as $T\rightarrow+\infty$ by Lemma \ref{lemma: claim2}.

For the second term in \eqref{eq: trans_R3}, by Remark \ref{remark: C_order}, we have
$$\begin{aligned}
e^{-rT}C_6(\hat{H}_T(y))Y_T(y)^{r_2}
&= e^{-rT}\bigg( O(\hat{H}_T(y)^{r_1\beta_1+r_2 + (\beta_2-\beta_1)}) +O(\hat{H}_T(y)^{r_1\beta_1+r_2}))  \bigg)Y_T(y)^{r_1} \\
&= e^{-rT}O\bigg( Y_T(y)^{\frac{\beta_2}{\beta_1-1}} \bigg) + e^{-rT}O\bigg( Y_T(y)^{\gamma_1}\bigg),
\end{aligned}$$
which also vanishes as $T\rightarrow+\infty$ by Lemma \ref{lemma: claim2}.
Therefore, we get the desired result.
\end{proof}

\begin{lemma}\label{lemma: 4_2}
For any $T>0$, we have
\begin{align}\label{desclm}
\lim\limits_{n\rightarrow+\infty} \E\big[e^{-r\tau_n}v(Y_{\tau_n}(y), \hat{H}_{\tau_n}(y)) \mathbf{1}_{\{T> \tau_n\}} \big] = 0.
\end{align}
\end{lemma}
\begin{proof}
By the definition of $\tau_n$, for all $t\leq\tau_n$, we have $Y_t(y) \in \big[\frac1n, n\big]$, and thus
$$\hat{H}_t(y) \leq \max(h, (1-\lam)^{-\frac{\beta_1}{\beta_1-1}}n^{1-\beta_1}) = O(1) + O(n^{1-\beta_1}). $$
Therefore, we have that $Y_t(y)^{r_1} \leq n^{r_1}$, $Y_t(y)^{r_2} \leq \big(\frac1n \big)^{r_2}= n^{-r_2}$.
Together with the fact that $r_1>0>\max\{\gamma_1,\gamma_2\}\geq \min\{\gamma_1,\gamma_2\} >r_2$ by Assumption (A1), we show the order of $v(Y_{\tau_n}(y), \hat{H}_{\tau_n}(y))$ in cases when $c_{\tau_n}^*=0$, $0<c_{\tau_n}^*<\hat{H}_{\tau_n}(y)$, and $c_{\tau_n}^* = \hat{H}_{\tau_n}(y)$.

Similar to the proof of Lemma \ref{lemma: 4}, if $c_{\tau_n}^*=0$, we have that
$$\begin{aligned}
v(Y_{\tau_n}(y), \hat{H}_{\tau_n}(y))
&= C_2(\hat{H}_{\tau_n}(y))Y_{\tau_n}(y)^{r_2} - \frac{k}{r\beta_2}(\lam \hat{H}_{\tau_n}(y))^{\beta_2} \\
&= O(1) + O(n^{\beta_2(1-\beta_1)})+ O(n^{-\gamma_1}) + O(n^{-\gamma_2}) \\
&= O(n^{-r_2}).
\end{aligned}$$
If $0<c_{\tau_n}^*<\hat{H}_{\tau_n}(y)$, we have that
$$\begin{aligned}
v(Y_{\tau_n}(y),\hat{H}_{\tau_n}(y)) = &
C_3(\hat{H}_{\tau_n}(y)) Y_{\tau_n}(y)^{r_1} + C_4(\hat{H}_{\tau_n}(y)) Y_{\tau_n}(y)^{r_2}\\
&~~+ \frac{2}{\kappa^2\gamma_1(\gamma_1-r_1)(\gamma_1-r_2)}Y_{\tau_n}(y)^{\gamma_1} - \frac{\lam \hat{H}_{\tau_n}(y)}{r}Y_{\tau_n}(y) \\
&= O(1) + O(n^{\beta_2(1-\beta_1)}) + O(n^{\beta_1(1-\beta_1)}) + O(n^{\frac{\beta_2}{1-\beta_1}}) + O(n^{-\gamma_1}) \\
&= O(n^{-r_2}).
\end{aligned}$$
If $c_{\tau_n} = \hat{H}_{\tau_n}(y)$, we have that
$$\begin{aligned}
v(Y_{\tau_n}(y), \hat{H}_{\tau_n}(y))
&=  C_5(\hat{H}_{\tau_n}(y))Y_{\tau_n}(y)^{r_1} + C_6(\hat{H}_{\tau_n}(y))Y_{\tau_n}(y)^{r_2} \\
&~~+ \frac{1}{r\beta_1}((1-\lam)\hat{H}_{\tau_n}(y))^{\beta_1} - \frac{1}{r}\hat{H}_{\tau_n}(y)Y_{\tau_n}(y)\\
&= O(n^{-r_2}).
\end{aligned}$$
In conclusion, in all the cases, $v(Y_{\tau_n}(y), \hat{H}_{\tau_n}(y)) = O(n^{-r_2})$.
In addition, similar to the proof of (A.25) in \cite{GuasoniHubermanR2020MFF}, there exists some constant $C$ such that
$\E[\mathbf{1}_{\{\tau\leq T\}}] \leq n^{-2\xi}(1+y^{2\xi})e^{CT}$, for any $\xi\geq 1$.
Putting all the pieces together, we get the desired claim \eqref{desclm}.
\end{proof}

\begin{lemma}\label{lemma: claim1}
For $\beta \in \{\beta_1,\beta_2\}$, we have
\begin{equation}
\lim\limits_{T\rightarrow+\infty}\E\bigg[e^{-rT}\hat{H}_T(y)^{\beta}\bigg] = 0.
\end{equation}
\end{lemma}
\begin{proof}
It is obvious that
$$\begin{aligned}
e^{-rT}\E\bigg[\hat{H}_T(y)^{\beta}\bigg]
\leq e^{-rT}\E\bigg[ \sup\limits_{s\leq T}Y_s(y)^{\frac{\beta}{\beta_1-1}}(1-\lam)^{-\frac{\beta_1\beta}{\beta_1-1}} \bigg] + e^{-rT}\E[h^{\beta} ], \\
    \end{aligned}$$
and it is clear that $e^{-rT}\E[h^{\beta} ] = O(e^{-rT})$ as $T\rightarrow+\infty$.

Define $W_t^{(\frac12\kappa)} := W_t + \frac12\kappa t$ with its running maximum $\bigg(W_t^{(\frac12\kappa)}\bigg)^*$.
It follows that
    $$\begin{aligned}
&e^{-rT}\E\bigg[ \sup\limits_{s\leq T}Y_s(y)^{\frac{\beta}{\beta_1-1}}(1-\lam)^{-\frac{\beta_1\beta}{\beta_1-1}} \bigg] \\
=& e^{-rT}O\bigg( \E\bigg[ \exp\bigg\{ aW_T^{(\zeta)} + b\bigg( W_T^{(\zeta)}\bigg)^* \mathbb{I}{\bigg\{\bigg( W_T^{(\zeta)}\bigg)^* \geq k\bigg\}} \bigg\} 
\bigg] \bigg),
\end{aligned}$$
where $a=0$, $b = -\frac{\beta}{\beta_1-1}\kappa > 0$, $\zeta = \frac12\kappa > 0$, and $k = 0$.
Note that $2a+b+2\zeta > 2a+b+\zeta > 0$, thanks to the Corollary A.7 in Guasoni et al. \cite{GuasoniHubermanR2020MFF}, we have that
$$\begin{aligned}
&\E\bigg[ \exp\bigg\{ aW_T^{(\zeta)} + b\bigg( W_T^{(\zeta)}\bigg)^*  \bigg\} 
\mathbb{I}{\bigg\{\bigg( W_T^{(\zeta)}\bigg)^* \geq k\bigg\}}
\bigg]
\\
=& \frac{2(a+b+\zeta)}{2a+b+2\zeta}\exp\bigg\{ \frac{(a+b)(a+b+2\zeta)}{2}T \bigg\} \Phi\bigg((a+b+\zeta)\sqrt T - \frac{k}{\sqrt T} \bigg) \\
+& \frac{2(a+\zeta)}{2a+b+2\zeta}\exp\bigg\{ (2a+b+2\zeta)k + \frac{a(a+2\zeta)}{2}T \bigg\} \Phi\bigg( -(a+\zeta)\sqrt T - \frac{k}{\sqrt T} \bigg),
\end{aligned}$$
and therefore
$$\begin{aligned}
&\lim\limits_{T\rightarrow+\infty}\frac1T\log \E\bigg[ \exp\bigg\{ aW_T^{(\zeta)} + b\bigg( W_T^{(\zeta)}\bigg)^*  \bigg\} 
\mathbb{I}{\bigg\{\bigg( W_T^{(\zeta)}\bigg)^* \geq k\bigg\}}
\bigg] -r \\
=& \frac{(a+b)(a+b+2\xi)}{2} -r = \frac{\kappa^2}{2}\gamma_0(\gamma_0-1) -r
  = \frac{\kappa^2}{2}(\gamma_0-r_1)(\gamma_0-r_2),
\end{aligned}$$
where $\gamma_0 = \frac{\beta}{\beta_1-1}$.
It thus holds that
$$
\begin{aligned}
&e^{-rT} \E\bigg[ \exp\bigg\{ aW_T^{(\zeta)} + b\bigg( W_T^{(\zeta)}\bigg)^* \mathbb{I}{\bigg\{\bigg( W_T^{(\zeta)}\bigg)^* \geq k\bigg\}} \bigg\} \\
=& \exp\bigg\{\bigg(\frac1T \log \E\bigg[ \exp\bigg\{ aW_T^{(\zeta)} + b\bigg( W_T^{(\zeta)}\bigg)^*  \bigg\} 
\mathbb{I}{\bigg\{\bigg( W_T^{(\zeta)}\bigg)^* \geq k\bigg\}}
\bigg] -r\bigg)T \bigg\} \\
=& O\bigg( \exp\bigg\{ \frac{\kappa^2}{2}(\gamma_0-r_1)(\gamma_0-r_2) T \bigg\} \bigg),
\end{aligned}
$$
as $T\rightarrow+\infty$.
Thanks to Assumption (A1), we have $r_1>0>\gamma_0 \geq \min(\gamma_1, \gamma_2) > r_2$.
It follows that $(\gamma_0-r_1)(\gamma_0-r_2) < 0$ and thus
$$\begin{aligned}
\E\bigg[e^{-rT}\hat{H}_T(y)^{\beta}\bigg]
&= O\bigg( \exp\bigg\{ \frac{\kappa^2}{2}(\gamma_0-r_1)(\gamma_0-r_2) T \bigg\} \bigg) + O(e^{-rT}), \\
\end{aligned}$$
which tends to 0 as $T\rightarrow+\infty$.
\end{proof}

\begin{lemma}\label{lemma: claim2}
For any $r_2<\gamma<r_1$, we have
\begin{equation}
\lim\limits_{T\rightarrow+\infty}\E\bigg[e^{-rT}Y_T(y)^{\gamma} \bigg] = 0.
\end{equation}
\end{lemma}

\begin{proof}
In fact, we have that
$$\begin{aligned}
\E\bigg[e^{-rT}Y_T(y)^{\gamma} \bigg]
&= e^{-rT} \E\bigg[ (ye^{rT}\cdot e^{-(r+\frac{\kappa^2}{2})T - \kappa W_T})^{\gamma} \bigg] \\
&= y^\gamma e^{-rT} \E\big[ e^{\gamma(-\frac{\kappa^2}{2}T - \kappa W_T)}\big] = O\bigg( e^{(\gamma-r_1)(\gamma-r_2)\frac{\kappa^2}{2}T} \bigg),
\end{aligned}$$
which converges to 0 in view that $r_2<\gamma<r_1$ by Assumption (A1).
\end{proof}

\subsection{Proof of Corollary \ref{cor: beta<1}}\label{sec:proofC3.1}
To conclude the main results in Corollary \ref{cor: beta<1}, it is sufficient to prove that the SDE \eqref{eq: SDE} has a unique strong solution $(X_t^*, H_t^*)$ for any initial value $(x,h) \in \mathcal{C}$. To this end, we can essentially follow the arguments in the proof of Proposition 5.9 in \cite{DengLiPY2020arXiv}. However, due to more complicated expressions of $C_2(h)$-$C_6(h)$ in \eqref{paraC} and different feedback functions, we need to prove the following auxiliary lemmas to conclude Corollary \ref{cor: beta<1}.

\begin{lemma}\label{lemma: 5}
The function $f$ is $C^1$ within each of the subsets of $\mathbb{R}_+^2: x\leq x_1(h),$ $x_1(h)<x<x_2(h)$ and $x_2(h)\leq x\leq x_3(h)$, and it is continuous at the boundary of $x = x_2(h)$ and $x=x_3(h)$. Moreover, we have that
{\small
\begin{equation}\label{eq: f_x}
\begin{aligned}
f_x(x,h) &= \frac{1}{g(y,h)} \\
&=\begin{cases}
\bigg(-C_2(h)r_2(r_2-1)(f(x,h))^{r_2-2} \bigg)^{-1}, & \mbox{if}~ x\leq x_\mathrm{zero}(h), \\
\begin{aligned}
\bigg(&-C_3(h)r_1(r_1-1)(f(x,h))^{r_1-2} \\
&- C_4(h)r_2(r_2-1)(f(x,h))^{r_2-2} \\
&- \frac{2(\gamma_1-1)}{\kappa^2(\gamma_1-r_1)(\gamma_1-r_2)}(f(x,h))^{\gamma_1-2} \bigg)^{-1},
\end{aligned}
& \mbox{if } x_\mathrm{zero}(h) < x < x_\mathrm{aggr}(h),  \\
\begin{aligned}
\bigg(&-C_5(h)r_1(r_1-1)(f(x,h))^{r_1-2} \\
&-C_6(h)r_2(r_2-1)(f(x,h))^{r_2-2} \bigg)^{-1},
\end{aligned}
& \mbox{if } x_\mathrm{aggr}(h) \leq x \leq x_\mathrm{lavs}(h),
\end{cases}
\end{aligned}
\end{equation}
}
and
\begin{equation}\label{eq: f_h}
f_h(x,h) = -g_h(f(x,h), h)\cdot f_x(x,h).
\end{equation}
\end{lemma}
\begin{proof}
The proof is the same as Lemma 5.6 in \cite{DengLiPY2020arXiv}, so we omit it.
\end{proof}

\begin{lemma}\label{lemma: pi_Lip}
The function $\pi^*$ is Lipschitz on $\mathcal{C}$.
\end{lemma}
\begin{proof}
By \eqref{eq: c_y_beta<1}, \eqref{eq: pi_y_beta<1} and the inverse transform, we can express $c^*$ and $\pi^*$ in terms of the primal variables as in \eqref{eq: c_beta<1} and \eqref{eq: pi_beta<1}.
Combining the expressions of $c^*$ and $\pi^*$ with Proposition \ref{prop: dual_solution_beta<1} which implies that the coefficients $(C_i)_{2\leq i\leq5}$ are $C^1$,
Lemma \ref{lemma: 5} which implies that the $C^1$ regularity of $f$, together with the continuity of $f$ at the boundary between the three regions, we can draw the conclusion that $(x,h)\rightarrow c^*(x,h)$ and $(x,h)\rightarrow \pi^*(x,h)$ are locally Lipschitz on $\mathcal{C}$.


\emph{(i)} Boundedness of $\frac{\partial\pi^*}{\partial x}$.

First using $\pi^*$ in \eqref{eq: pi_beta<1}, we have
\begin{equation}\label{eq: pi_d_beta<1}
\begin{aligned}
&\frac{\partial \pi^*}{\partial x}(x,h) = \frac{\mu-r}{\sigma^2} \\
\times&
\begin{cases}
1-r_2, & \mbox{if } x < x_\mathrm{zero}(h), \\
\begin{aligned}
&\bigg(\frac{2r}{\kappa^2} C_3(h)(r_1-1)f_2^{r_1-2}(x,h)\frac{\partial f_2}{\partial x} +  \frac{2r}{\kappa^2} C_4(h)(r_2-1)f_2^{r_2-2}(x,h)\frac{\partial f_2}{\partial x} \\
&+ \frac{2(\gamma_1-1)^2}{\kappa^2(\gamma_1-r_1)(\gamma_1-r_2)}f_2^{\gamma_1-2}(x,h)\frac{\partial f_2}{\partial x}\bigg),
\end{aligned}
& \mbox{if } x_\mathrm{zero}(h) \leq x \leq x_\mathrm{aggr}(h), \\
\begin{aligned}
\frac{2r}{\kappa^2} C_5(h)(r_1-1)f^{r_1-2}(x,h)\frac{\partial f_2}{\partial x} +  \frac{2r}{\kappa^2} C_6(h)(r_2-1)f^{r_2-2}(x,h)\frac{\partial f_2}{\partial x} ,
\end{aligned}
& \mbox{if } x_\mathrm{aggr}(h)< x\leq x_\mathrm{lavs}(h).
\end{cases}
\end{aligned}
\end{equation}
Note that the first line is constant and hence bounded.
For the second line, by differentiating \eqref{eq: f_2_beta<1} and using the fact that $r_1(r_1-1) = r_2(r_2-1) = \frac{2r}{\kappa^2}$, we have that
$$\begin{aligned}
1 = &-\frac{2r}{\kappa^2}C_3(h)f_2(x,h)^{r_1-2}\frac{\partial f_2}{\partial x}(x,h)
-\frac{2r}{\kappa^2}C_4(h)f_2(x,h)^{r_2-2}\frac{\partial f_2}{\partial x}(x,h) \\
&-\frac{2(\gamma_1-1)}{\kappa^2(\gamma_1-r_1)(\gamma_1-r_2)}f_2(x,h)^{\gamma_1-2}\frac{\partial f_2}{\partial x}(x,h).
\end{aligned}$$
Plugging this back to $\frac{\partial \pi^*}{\partial x}$, we can obtain
\begin{align*}
\frac{\partial \pi^*}{\partial x}(x,h)
= \frac{\mu-r}{\sigma^2}\bigg\{ \bigg(\frac{2r}{\kappa^2}C_3(h)f_2^{r_1-1}(x,h)(r_1-r_2) + \frac{2(\gamma_1-1)}{\kappa^2(\gamma_1-r_1)}f_2^{\gamma_1-1}(x,h)\bigg)\frac{1}{f_2}\frac{\partial f_2}{\partial x} + (1-r_2) \bigg\}
\end{align*}
Combining with Lemma \ref{lemma: 5}, we can obtain
$\frac{\partial \pi^*}{\partial x}(x,h) = \frac{\mu-r}{\sigma}(\frac{A_1}{B_1}+(1-r_2))$,
where
\begin{equation}\label{eq: A1B1_def}\begin{aligned}
A_1 &:= \frac{2r}{\kappa^2}C_3(h)f_2^{r_1-1}(x,h)(r_1-r_2) + \frac{2(\gamma_1-1)}{\kappa^2(\gamma_1-r_1)}f_2^{\gamma_1-1}(x,h), \\
B_1 &:= -\frac{2r}{\kappa^2}C_3(h)f_2(x,h)^{r_1-1} - \frac{2r}{\kappa^2}C_4(h)f_2(x,h)^{r_2-1} - \frac{2(\gamma_1-1)}{\kappa^2(\gamma_1-r_1)(\gamma_1-r_2)}f_2(x,h)^{\gamma_1-1}.
\end{aligned}\end{equation}

We next show that $A_1>0$ and $B_1<0$, and $\frac{A_1}{B_1}$ is bounded.
We only need to discuss the case that $y_1(h) > y_2(h) = ((1-\lam)h)^{\beta_1-1}$, because the second region reduces to a point for any fixed $h$ if $y_1(h) = y_2(h)$.
Indeed, it is obvious that $A_1>0$ since $C_3(h)>0$ according to the proof of Lemma \ref{lemma: beta<1} and $\gamma_1<0$.
Moreover, we have that
$$\begin{aligned}
A_1 &= \frac{2}{\kappa^2}\bigg( r(r_1-r_2)C_3(h)+\frac{\gamma_1-1}{\gamma_1-r_1}f_2^{\gamma_1-r_1}(x,h) \bigg)f_2^{r_1-1}(x,h) \\
&= \frac{2}{\kappa^2}\bigg( \frac{r_2}{\gamma_1-r_1}y_1(h)^{\gamma_1-r_1} + \frac{\lam}{(1-\lam)^{(\gamma_1-1)(\beta_1-1)}} y_2(h)^{\gamma_1-1} y_1(h)^{1-r_1} +\frac{\gamma_1-1}{\gamma_1-r_1}f_2^{\gamma_1-r_1}(x,h)\bigg)f_2^{r_1-1}(x,h)\\
&\leq  K_1\big(y_1(h)^{1-r_1}y_2(h)^{\gamma_1-1}f_2(x,h)^{r_1-1} + f_2(x,h)^{\gamma_1-1} ),
\end{aligned}$$
where $K_1$ is some positive constant.
For $B_1$, according to the proof of Lemma \ref{lemma: beta<1}, we have that
$$\begin{aligned}
B_1 &= -\frac{2r}{\kappa^2}C_3(h)f_2(x,h)^{r_1-1} - \frac{2r}{\kappa^2}C_4(h)f_2(x,h)^{r_2-1} - \frac{2(\gamma_1-1)}{\kappa^2(\gamma_1-r_1)(\gamma_1-r_2)}f_2(x,h)^{\gamma_1-1} \\
&\leq -\frac{2r}{\kappa^2}C_3(h)f_2(x,h)^{r_1-1} - Cf_2(x,h)^{\gamma_1-1} \\
&\leq - K_2 ( y_1(h)^{1-r_1}y_2(h)^{\gamma_1-1}f_2(x,h)^{r_1-1} + f_2(x,h)^{\gamma_1-1}),
\end{aligned}$$
where $K_2$ is some positive constant. Therefore, $0 > \frac{A_1}{B_1} \geq -C$ for some positive constant $C$ independent of $h$, and thus $\frac{\partial \pi^*}{\partial x}$ in the second line of \eqref{eq: pi_d_beta<1} is bounded.

For the third line, by differentiating \eqref{eq: f_3_beta<1} and using the fact that $r_1(r_1-1) = r_2(r_2-1) = \frac{2r}{\kappa^2}$, we have that
$1 = -\frac{2r}{\kappa^2}C_5(h)f_3(x,h)^{r_1-2}\frac{\partial f_3}{\partial x}(x,h)
-\frac{2r}{\kappa^2}C_6(h)f_3(x,h)^{r_2-2}\frac{\partial f_3}{\partial x}(x,h)$. Putting this back to the third line of \eqref{eq: pi_d_beta<1}, we can obtain
$\frac{\partial\pi^*}{\partial x}(x,h)
= \frac{\mu-r}{\sigma^2}\{ \frac{A_2}{B_2} + (1-r_2)\},
$
where
\begin{equation}\label{eq: A2B2_def}\begin{aligned}
A_2 &:= \frac{2r(r_1-r_2)}{\kappa^2}C_5(h)f_3^{r_1-1}(x,h), \\
B_2 &:= -\frac{2r}{\kappa^2}C_5(h)f_3(x,h)^{r_1-1} - \frac{2r}{\kappa^2}C_6(h)f_3(x,h)^{r_2-1},
\end{aligned}\end{equation}
by combining with the results of Lemma \ref{lemma: 5}.
In fact, by the proof of Lemma \ref{lemma: beta<1}, we have $C_5(h)>0$ and $C_6(h)>0$, therefore, $A_2>0$ and $B_2<0$.
Moreover, we have that
$$
B_2 = -\frac{2r}{\kappa^2}C_5(h)f_3(x,h)^{r_1-1} - \frac{2r}{\kappa^2}C_6(h)f_3(x,h)^{r_2-1}
\leq -\frac{2r}{\kappa^2}C_5(h)f_3(x,h)^{r_1-1},
$$
and thus
$0 > \frac{A_2}{B_2} \geq r_2-r_1$, indicating that $\frac{\partial \pi^*}{\partial x}$ is bounded in the third line.\\

\emph{(ii)} Boundedness of $\frac{\partial\pi^*}{\partial h}$.

First, using equations \eqref{eq: f_x} and \eqref{eq: f_h} and the definition of $g(\cdot,h) = -v_y(\cdot,h)$, we have
{\small
$$\begin{aligned}
&f_h(x,h) = -g_h(f,h)\cdot f_x(x,h) \\
=& \begin{cases}
C_2'(h)r_2f_1(x,h)^{r_2-1}\cdot\bigg(-\frac{2r}{\kappa^2}C_2(h)f_1(x,h)^{r_2-2}\bigg)^{-1}, &\mbox{if } x< x_{\mathrm{zero}}(h),\\
\begin{aligned}
&\bigg( C_3'(h)r_1f_2(x,h)^{r_1-1} + C_4'(h)r_2f_2(x,h)^{r_2-1} -\frac{\lam}{r} \bigg) \\
\times&\bigg(-\frac{2r}{\kappa^2}C_3(h)f_2(x,h)^{r_1-2}
-\frac{2r}{\kappa^2}C_4(h)f_2(x,h)^{r_2-2} \\
&-\frac{2(\gamma_1-1)}{\kappa^2(\gamma_1-r_1)(\gamma_1-r_2)}f_2(x,h)^{\gamma_1-2}\bigg)^{-1},
\end{aligned} &\mbox{if } x_{\mathrm{zero}}(h) \leq x \leq x_\mathrm{aggv}(h),\\
\begin{aligned}
&\bigg( C_5'(h)r_1f_3(x,h)^{r_1-1} + C_6'(h)r_2f_3(x,h)^{r_2-1} -\frac{1}{r} \bigg) \\
\times&\bigg( -\frac{2r}{\kappa^2}C_5(h)f_3(x,h)^{r_1-2}
-\frac{2r}{\kappa^2}C_6(h)f_3(x,h)^{r_2-2} \bigg)^{-1},
\end{aligned}
 &\mbox{if } x_\mathrm{aggv}(h) \leq x \leq x_\mathrm{lavs}(h).
\end{cases}
\end{aligned}$$
}We analyze the derivative $\frac{\partial \pi^*}{\partial h}$ in different regions separately. In the region $x < x_\mathrm{zero}(h)$, $\frac{\partial \pi^*}{\partial h} = 0$, hence it is bounded. In the region $x_\mathrm{zero}(h) \leq x(h) \leq x_\mathrm{aggv}(h)$, we also only need to discuss the case that $y_1(h) > y_2(h) = ((1-\lam)h)^{\beta_1-1}$, and
$$\begin{aligned}
\frac{\partial \pi^*}{\partial h}
= \frac{\mu-r}{\sigma^2}\bigg( &\frac{2r}{\kappa^2}C_3'(h)f_2(x,h)^{r_1-1} + \frac{2r}{\kappa^2}C_3(h)(r_1-1)f_2(x,h)^{r_2-2}\frac{\partial f_2}{\partial h}  \\
+& \frac{2r}{\kappa^2}C_4'(h)f_2(x,h)^{r_2-1} + \frac{2r}{\kappa^2}C_4(h)(r_2-1)f_2(x,h)^{r_2-2}\frac{\partial f_2}{\partial h}\\
+& \frac{2(\gamma_1-1)^2}{\kappa^2(\gamma_1-r_1)(\gamma_1-r_2)}f_2^{\gamma_1-2}(x,h)\frac{\partial f_2}{\partial h}
\bigg).
\end{aligned}$$
By differentiating \eqref{eq: f_2_beta<1} and using the fact that $r_1(r_1-1) = r_2(r_2-1) = \frac{2r}{\kappa^2}$, we have that
$$\begin{aligned}
&C_4'(h)\frac{2r}{\kappa^2}f_2(x,h)^{r_2-1} + C_4(h)\frac{2r}{\kappa^2}(r_2-1)\frac{\partial f_2}{\partial h}f_2(x,h)^{r_2-2}\\
=& -C_3'(h)r_1(r_2-1)f_2(x,h)^{r_1-1} - C_3(h)\frac{2r}{\kappa^2}(r_2-1)\frac{\partial f_2}{\partial h}f_2(x,h)^{r_1-2} \\
&- \frac{2(\gamma_1-1)(r_2-1)}{\kappa^2(\gamma_1-r_1)(\gamma_1-r_2)}\frac{\partial f_2}{\partial h}f_2(x,h)^{\gamma_1-2} + (r_2-1)\frac{\lam}{r}.
\end{aligned}$$
Putting this back to the previous expression of $\frac{\partial \pi^*}{\partial h}$, we can obtain that
\begin{equation}\label{eq: pi_h_R2}\begin{aligned}
\frac{\partial \pi^*}{\partial h}
&= \frac{\mu-r}{\sigma^2}\bigg(
\bigg[\frac{2r}{\kappa^2}(r_1-r_2)C_3(h)f_2(x,h)^{r_1-1} - \frac{2(\gamma_1-1)(r_2-1)}{\kappa^2(\gamma_1-r_1)(\gamma_1-r_2)} f_2(x,h)^{\gamma_1-1} \bigg]\frac{1}{f_2}\frac{\partial f_2}{\partial h} \\
&+ r_1(r_1-r_2)C_3'(h)f_2(x,h)^{r_1-1} + (r_2-1)\frac{\lam}{r}
\bigg) \\
&= \frac{\mu-r}{\sigma^2}\bigg( A_1\cdot \frac{1}{f_2}\frac{\partial f_2}{\partial h}
+ r_1(r_1-r_2)C_3'(h)f_2(x,h)^{r_1-1} - \frac{\lam}{r}(1-r_2) \bigg),
\end{aligned}\end{equation}
where $A_1$ is defined in \eqref{eq: A1B1_def}.
In \eqref{eq: pi_h_R2}, the third term is a constant.
For the second term, by the proof of Lemma \ref{lemma: beta<1}, we first have $C_3'(h)>0$, and
$$\begin{aligned}
C_3'(h) &= \frac{1}{r(r_1-r_2)h}(kr_2(\lam h)^{\beta_2} + \lam hr_1 w(h)^{\beta_1-1})w(h)^{-r_1(\beta_1-1)} \\
&= \frac{1}{r(r_1-r_2)}\bigg(\lam(r_1+r_2\beta_2)y_1(h)^{1-r_1} + \frac{r_2\beta_2}{\gamma_1h}y_1(h)^{\gamma_1-r_1} \bigg).
\end{aligned}$$
In the sequel of the proof below, let $K_1>0$ be a generic positive constant independent of $(x,h)$, which may be different from line to line.
Hence, we have that
$$\begin{aligned}
r_1(r_1-r_2)C_3'(h)f_2(x,h)^{r_1-1}
&= \bigg(\lam(r_1+r_2\beta_2)y_1(h)^{1-r_1} + \frac{r_2\beta_2}{\gamma_1h}y_1(h)^{\gamma_1-r_1} \bigg)f_2(x,h)^{r_1-1} \\
&\leq \bigg(\lam(r_1+r_2\beta_2)y_1(h)^{1-r_1} + \frac{r_2\beta_2}{\gamma_1h}y_1(h)^{\gamma_1-r_1} \bigg)y_1(x,h)^{r_1-1} \\
&\leq K_1\bigg( 1 + \bigg( \frac{y_2(h)}{y_1(h)}\bigg)^{1-\gamma_1} \bigg)
\leq K_1.
\end{aligned}$$
Therefore, the second term is bounded.

For the first term in \eqref{eq: pi_h_R2}, by virtue of $A_1 \leq C(y_1(h)^{1-r_1}y_2(h)^{\gamma_1-1}f_2(x,h)^{r_1-1} + f_2(x,h)^{\gamma_1-1})$, it is enough to show that $\frac{1}{f_2}\frac{\partial f_2}{\partial h} \geq C(y_1(h)^{1-r_1}y_2(h)^{\gamma_1-1}f_2(x,h)^{r_1-1} + f_2(x,h)^{\gamma_1-1})$ for some positive constant $C$.
Indeed, we have that
$\frac{1}{f_2}\frac{\partial f_2}{\partial h}
= \left( C_3'(h)r_1f_2(x,h)^{r_1-1} + C_4'(h)r_2f_2(x,h)^{r_2-1} - \frac{\lam}{r} \right) \times \frac{1}{B_1}$,
where $B_1$ is defined in \eqref{eq: A1B1_def}.
As $C_3'(h)r_1f_2(x,h)^{r_1-1}$ and $\frac{\lam}{r}$ are bounded, it is sufficient to show that $C_4'(h)r_2f_2(x,h)^{r_2-1}$ is bounded.
As $C_6'(h) = -C_5'(h)y_3(h)^{r_1-r_2}$, let $K_2$ be a generic constant that may differ from line to line, we have that
$$\begin{aligned}
|C_4'(h)| &= |C_4'(h) - C_6'(h) + C_6'(h)| \\
&= \bigg|\frac{(\gamma_1-r_2)(\beta_1-1)}{(r_2-r_1)(1-\beta_1)(\gamma_1-r_2)}(1-\lam)^{(\gamma_1-r_2)(\beta_1-1)}h^{r_1(\beta_1-1)}
- C_5'(h)y_3(h)^{r_1-r_2}\bigg| \\
&= \bigg| K_2h^{r_1(\beta_1-1)} - \big(C_3'(h) - (C_3'(h)-C'_5(h))\big)y_3(h)^{r_1-r_2}\bigg| \\
&= | K_2h^{r_1(\beta_1-1)} - C_3'(h)y_3(h)^{r_1-r_2}| \\
&\leq  K_2h^{r_1(\beta_1-1)} + K_2\big(y_1(h)^{1-r_1} + y_2(h)^{1-\gamma_1}y_1(h)^{\gamma_1-r_1}\big)y_2(h)^{r_1-r_2} \\
&\leq K_2\big( y_2(h)^{r_1} + y_2(h)^{r_1-r_2}y_1(h)^{1-r_1} + y_2(h)^{2r_1-\gamma_1}y_1(h)^{\gamma_1-r_1} \big),
\end{aligned}$$
where the first inequality holds because of $y_3(h) = (1-\lam)y_2(h)$.
Similarly, let $K_3$ be the constant that may differ from line to line, it follows that
$$\begin{aligned}
|C_4'(h)r_2(f_2(x,h)^{r_2-1})|
&=  K_3|C'_4(h)|f_2(x,h)^{r_2-1} \\
&\leq  K_3|C'_4(h)|y_2(h)^{r_2-1} \\
&\leq K_3\big( y_2(h)^{r_1} + y_2(h)^{r_1-r_2}y_1(h)^{1-r_1} + y_2(h)^{2r_1-\gamma_1}y_1(h)^{\gamma_1-r_1} \big)y_2(h)^{r_2-1}\\
&= K_3\bigg( 1 + \bigg(\frac{y_2(h)}{y_1(h)}\bigg)^{r_1-1} + \bigg( \frac{y_2(h)}{y_1(h)}\bigg)^{r_1-\gamma_1} \bigg),
\end{aligned}$$
which is bounded as $y_1(h)>y_2(h)$.

In the region $x_\mathrm{aggv}(h) \leq x(h) \leq x_\mathrm{lavs}(h)$, similar computations yield that $$\frac{\partial \pi^*}{\partial h}=\frac{\mu-r}{\sigma^2}\bigg( A_2\cdot \frac{1}{f_3}\frac{\partial f_3}{\partial h}
+ r_1(r_1-r_2)C_5'(h)f_3(x,h)^{r_1-1} - \frac{1}{r}(1-r_2) \bigg),$$
where $A_2$ is defined in \eqref{eq: A2B2_def}.
For the term $r_1(r_1-r_2)C_5'(h)f_3(x,h)^{r_1-1}$, due to $C_5'(h) < C_3'(h)$, we have that
$r(r_1-r_2)C_5'(h)f_3(x,h)^{r_1-1}
\leq r(r_1-r_2)C_3'(h)f_2(\bar{x},h)^{r_1-1},$
which is bounded as $f_3(x,h) < y_2(h) \leq f_2(\bar{x},h)$, where $\bar{x}$ is chosen such that $y_2(h) \leq f_2(\bar{x},h) \leq y_1(h)$.
Moreover, for the term $A_2\cdot \frac{1}{f_3}\frac{\partial f_3}{\partial h}$,
similar to the proof in the region $y_2(h)\leq f_2(x,h)\leq y_1(h)$, it is enough to check that $C_6'(h)f_3(x,h)^{r_2-1}$ is bounded.
Indeed, we can obtain
$$\begin{aligned}
|C_6'(h)f_3(x,h)^{r_2-1}|
= |C_5'(h)y_3(h)^{r_1-r_2}f_3(x,h)^{r_2-1}|
\leq C_5'(h)f_3(x,h)^{r_1-1},
\end{aligned}$$
which is shown to be bounded. Putting all the pieces together completes the proof.
\end{proof}

%


\subsection{Proof of Proposition \ref{prop: samevalue} (Concavification Principle)}\label{sec: proof_prop_samevalue}
To prove this proposition, we claim that under the optimal controls $c_t^*$ and $\pi_t^*$, it holds that $\tU(c_t^*, H_t^*) = U(c_t^* - \lam H_t^*)$ all the time.
In fact, for any $(x,h) \in\mathcal{C}$, according to the definition of concave envelop $\tU(x,h)$ of $U^*(x,h)$ in $x\in[0,h]$ in \eqref{eq: ce_beta<1}, we can easily see that $\tU(x,h) = U^*(x,h)$ if $x\in \mathcal{C}_h := \{0\}\cup[z(h), h]$, where $z(h)$ is defined in Section \ref{sec: ce}.
We verify the claim in all the regions of the wealth $X_t^*$.

If $X_t^* < x_\mathrm{zero}(H_{t}^*)$, then $c_t^* = 0 \in\mathcal{C}_{H_t^*}$, indicating that $\tU(c_t^*, H_t^*) = U(c_t^*-\lam H_t^*)$.

If $x_\mathrm{zero}(H_{t}^*) \leq X_t^* \leq x_\mathrm{aggr}(H_{t}^*)$, yielding the existence of the solution $z(H_t^*)$ for equation \eqref{eq: beta<1,tagent z} with $h = H_t^*$.
Moreover, the optimal consumption satisfies that $z(H_t^*) \leq c_t^* = \lam H_{t}^* + (f(X_t^*, H_{t}^*))^{\frac1{\beta_1-1}} \leq H_t^*$, where $f(x,h)$ is defined in Corollary \ref{cor: beta<1}.
This leads to the fact that $c_t^* \in \mathcal{C}_{H_t^*}$ and thus $\tU(c_t^*, H_t^*) = U(c_t^* - \lam H_t^*)$.

If $x_\mathrm{aggr}(H_t^*) < X_t^* \leq x_\mathrm{lavs}(H_t^*)$, then $c_t^* = H_t^* \in \mathcal{C}_{H_t^*}$, indicating that $\tU(c_t^*, H_t^*) = U(c_t^* - \lam H_t^*)$.

Therefore, we have verified that the optimal consumption rate $c_t^*$ always leads to  $U(c_t^* - \lam H_t^*) = \tU(c_t^*, H_t^*)$.
Thus, given the optimal portfolio $\pi^*_t$ and $c^*_t$ for the stochastic control problem \eqref{eq: concavevalue}, based on the fact that $\tU(x, h) \geq U(x-\lam h)$ everywhere and corresponding $\tu \geq u$, we have
\begin{equation*}
\tu(x,h) = \mathbb{E}\bigg[ \int_0^\infty e^{-rt}\tU(c_t^*, H_t^*)dt \bigg] = \mathbb{E}\bigg[ \int_0^\infty e^{-rt}U(c_t^* - \lam H_t^*)dt \bigg] \leq u(x,h) \leq \tu(x,h),
\end{equation*}
that is, $\tu = u$, and the optimal portfolio and consumption for \eqref{eq: primalvalue} are the same as \eqref{eq: concavevalue}.

\subsection{Proof of Lemma \ref{lemma: beta<1}}\label{sec: proof_lemma_beta<1}
We prove $v_yy(y,h)>0$ in three regions: $y>y_1(h)$, $y_2(h)\leq y\leq y_1(h)$, and $y_3(h)\leq y < y_2(h)$,  respectively.

(i) In the region $y_3(h) \leq y < y_2(h)$, $v_{yy}(y, h) = r_1(r_1-1)C_5(h)y^{r_1-2} + r_2(r_2-1)C_6(h)y^{r_2-2}$.

As $r_1(r_1-1) = r_2(r_2-1) = \frac{2r}{\kappa^2} > 0$,
we only need prove $C_5(h) > 0$ and $C_6(h)>0$.
We separate the proof into two cases: the case that $y_1(h) > y_2(h)$ and that $y_1(h) = y_2(h)$.
If $y_1(h) = w(h)^{\beta_1-1} > y_2(h) = ((1-\lam)h)^{\beta_1-1}$, we deduce that
$$\begin{aligned}
C_3(h) &= \frac{y_1(h)^{-r_1}}{r(r_1-r_2)}
\bigg( \frac{kr_2}{\beta_2}(\lam h)^{\beta_2} + \frac{r_1r_2}{(\gamma_1-r_1)\gamma_1}y_1(h)^{\gamma_1}+ \lam h r_1 y_1(h) \bigg),\\
&= \frac{w(h)^{-r_1(\beta_1-1)}}{r(r_1-r_2)}\bigg( \frac{r_2}{\gamma_1}w(h)^{\beta_1} + \lam hr_2 w(h)^{\beta_1-1} + \frac{r_1r_2}{(\gamma_1-r_1)\gamma_1}w(h)^{\beta_1} + \lam h r_1 w(h)^{\beta_1-1} \bigg) \\
&= \frac{w(h)^{-r_1(\beta_1-1)}}{r(r_1-r_2)}\bigg( \frac{r_2}{\gamma_1-r_1}w(h)^{\beta_1} + \lam h w(h)^{\beta_1-1} \bigg)>0,
\end{aligned}$$
and
$$\begin{aligned}
C_3(h) - C_5(h) &= \frac{y_2(h)^{-r_1}}{r(r_1-r_2)}\bigg( -\frac{r_2}{\beta_1}((1-\lam)h)^{\beta_1}+ \frac{r_1r_2}{\gamma_1(\gamma_1-r_1)}y_2(h)^{\gamma_1} - (1-\lam) h r_1y_2(h) \bigg) \\
&= \frac{1}{r(r_1-r_2)(1-\beta_1)(\gamma_1-r_1)} ((1-\lam)h)^{r_2\beta_1+r_1} < 0,
\end{aligned}$$
therefore, we have $C_5(h) = C_3(h) - (C_3(h)-C_5(h)) > 0$.

We next prove that $C_6'(h) < 0$, and hence $C_6(h) = -\int_h^\infty C'_6(s) ds > 0$. It is easy to see that $C_3'(h)-C_5'(h)<0$, and hence $C_5'(h) > C_3'(h)> 0$,
where the second inequality follows from
$$\begin{aligned}
C_3'(h)
&= \frac1{r(r_1-r_2)}\bigg(  \bigg(-\frac{k}{\beta_2}(\lam h)^{\beta_2}+\frac1{\gamma_1}w(h)^{\beta_1} + \lam h w(h)^{\beta_1-1} \bigg)r_1r_2(\beta_1-1)w(h)^{-r_1\beta_1-r_2}w'(h) \bigg) \\
&~~+ \frac1{r(r_1-r_2)h}(kr_2(\lam h)^{\beta_2} + \lam h r_1 w(h)^{\beta_1-1})w(h)^{-r_1(\beta_1-1)}> 0,
\end{aligned}$$
thanks to $\frac{k}{\beta_2}(\lam h)^{\beta_2} - \frac1{\gamma_1}w(h)^{\beta_1} - \lam h w(h)^{\beta_1-1} \leq 0$ and $w'(h) > 0$.

Along the free boundary condition \eqref{eq: free_boundary_dual}, we have
$C'_5(h)y_3(h)^{r_1} + C'_6(h)y_3(h)^{r_2}  = 0$,
therefore, we deduce that $C_6'(h) = -C_5'(h)y_3(h)^{r_1-r_2} < 0$.

We then consider the case that
$y_1(h) = y_2(h) = \frac{k}{\beta_2}\lam^{\beta_2}h^{\beta_2-1} +
\frac{1}{\beta_1}(1-\lam)^{\beta_1}h^{\beta_1-1} \leq ((1-\lam)h)^{\beta_1-1},$
in which we have that
$$\begin{aligned}
C_5(h) &= \frac{y_1(h)^{r_2-1}}{r(r_1-r_2)}\bigg( \frac{kr_2}{\beta_2}(\lam h)^{\beta_2} + \frac{r_2}{\beta_1}((1-\lam)h)^{\beta_1} + hr_1y_1(h) \bigg)= \frac{h}{r(r_1-r_2)}y_1(h)^{r_2} > 0,
\end{aligned}$$
and
$C_5'(h) = \frac{y_1(h)^{r_2-1}}{r(r_1-r_2)}\bigg( y_1(h) + r_2hy_1'(h) \bigg)  > 0$. Thus, it holds that $C_6'(h) = -C_5'(h)y_3(h)^{r_1-r_2} < 0$, implying that $C_6(h)>0$ when $y_1(h) = y_2(h)$.

(ii) In the region $y_2(h) \leq y \leq y_1(h)$, we only need to consider the case that $y_1(h) = w(h)^{\beta_1-1} > y_2(h) = ((1-\lam)h)^{\beta_1-1}$, otherwise the second-order derivative of $v(y,h)$ in $y$ is trivial because this region reduces to a point.

Because $C_3(h)>0$, $C_4(h) > C_4(h) - C_6(h)$, $r_1(r_1-1) = r_2(r_2-1) = \frac{2r}{\kappa^2}$,
we deduce that
$$\begin{aligned}
v_{yy}(y, h)
&= \frac{2r}{\kappa^2}\bigg( C_3(h)y^{r_1-\gamma_1} + C_4(h)y^{r_2-\gamma_1} + \frac{\gamma_1-1}{r(\gamma_1-r_1)(\gamma_1-r_2)} \bigg)y^{\gamma_1-2} \\
&> \frac{2r}{\kappa^2}\bigg( (C_4(h)-C_6(h))y^{r_2-\gamma_1} + \frac{\gamma_1-1}{r(\gamma_1-r_1)(\gamma_1-r_2)} \bigg)y^{\gamma_1-2} \\
&\geq \frac{2r}{\kappa^2}\bigg( (C_4(h)-C_6(h))((1-\lam)h)^{(r_2-\gamma_1)(\beta_1-1)} + \frac{\gamma_1-1}{r(\gamma_1-r_1)(\gamma_1-r_2)} \bigg)y^{\gamma_1-2},
\end{aligned}$$
where the last inequality holds because $y \geq ((1-\lam)h)^{\beta_1-1}$, $\gamma_1>r_2$, and $C_4(h)-C_6(h)<0$.
Moreover, we have that
$$\begin{aligned}
(C_4(h)-C_6(h))((1-\lam)h)^{(r_2-\gamma_1)(\beta_1-1)} + \frac{\gamma_1-1}{r(\gamma_1-r_1)(\gamma_1-r_2)}
= \frac{\gamma_1-1}{r(\gamma_1-r_1)(r_1-r_2)}
>0.
\end{aligned}$$
Thus, we deduce that $v_{yy}(y,h) > 0$.

(iii) In the region $y> y_1(h)$, $v_{yy}(y,h) = r_2(r_2-1)C_2(h)y^{r_2-2}$.
Since $r_2(r_2-1) = \frac{2r}{\kappa^2} > 0$, we only need to prove that $C_2(h)> 0$.
We also discuss $C_2(h)>0$ in two cases that $y_1(h) > y_2(h)$ or $y_1(h) = y_2(h)$ respectively.

If $y_1(h) > y_2(h)$, indicating that $
y_1(h) = w(h)^{\beta_1-1},
$ we have $\frac{k}{\beta_2}(\lam h)^{\beta_2}-\frac1{\gamma_1}w(h)^{\beta_1} - \lam h w(h)^{\beta_1-1} = 0$.
Similar to the proof of $C_5(h) > 0$, we have
$$\begin{aligned}
C_2(h) &> C_2(h)-C_6(h) = (C_2(h)-C_4(h)) + (C_4(h) - C_6(h)) \\
&= \frac{w(h)^{-r_2(\beta_1-1)}}{r(r_1-r_2)}\bigg( \frac{r_1}{\gamma_1-r_2}w(h)^{\beta_1} + \lam h w(h)^{\beta_1-1} \bigg) - \frac{1}{r(r_2-r_1)(1-\beta_1)(\gamma_1-r_2)} y_2(h)^{\gamma_1-r_2} \\
&> \frac{r_1}{r(r_1-r_2)(\gamma_1-r_2)}y_1(h)^{\gamma_1-r_2} + \frac{\gamma_1-1}{r(r_1-r_2)(\gamma_1-r_2)}y_2(h)^{\gamma_1-r_2} \\
&> \frac1{r(r_1-r_2)}y_2(h)^{\gamma_1-r_2}
> 0.
\end{aligned}$$
If $y_1(h) = y_2(h)$, similar to the proof of $C_5(h) > 0$, we can obtain that
$C_2(h)> C_2(h) - C_6(h) = \frac{h}{r(r_1-r_2)}y_1(h)^{r_1} > 0$.

\subsection{Proof of Corollary \ref{cor: asy_infty_wealth}}\label{sec: proof_asy_infty_wealth}
\begin{proof}

We first have that
$$\begin{aligned}
&\lim\limits_{h\rightarrow+\infty} C_6(h)h^{-r_2-r_1\beta_1}= \frac1{r_2+r_1\beta_1}\lim\limits_{h\rightarrow+\infty} \frac{C_6'(h)}{h^{r_1(\beta_1-1)}} = -\frac1{r_2+r_1\beta_1}(1-\lam)^{\beta_1(r_1-r_2)}\lim\limits_{h\rightarrow+\infty} \frac{C_5'(h)}{h^{r_2(\beta_1-1)}} \\
=& -\frac{\gamma_1-r_1}{\gamma_1-r_2}(1-\lam)^{\beta_1(r_1-r_2)}
\lim\limits_{h\rightarrow+\infty} C_5(h)h^{-r_1-r_2\beta_1},
\end{aligned}$$
by L'H$\hat{\mathrm{o}}$pital's rule.
To compute $\lim\limits_{h\rightarrow+\infty} C_5(h)h^{-r_1-r_2\beta_1}$, we need to consider two cases that $y_1(h) > y_2(h)$ and $y_1(h) = y_2(h)$.


We first consider the case that $y_1(h) = y_2(h)$ as $h\rightarrow+\infty$, indicating that $\beta_1>1-\lam$ and $\beta_2\leq\beta_1$ in condition (S2) or (S3), therefore,
$C_5(h) = \frac{h}{r(r_1-r_2)}y_1(h)^{r_2}$, and thus
$$\begin{aligned}
\lim\limits_{h\rightarrow+\infty} C_5(h)h^{-r_2\beta_1-r_1}
&= \frac{1}{r(r_1-r_2)}\lim\limits_{h\rightarrow+\infty}
\bigg( \frac{k}{\beta_2}\lam^{\beta_2}h^{\beta_2-\beta_1} + \frac{1}{\beta_1}(1-\lam)^{\beta_1} \bigg)^{r_2} \\
&= \frac{1}{r(r_1-r_2)}
\bigg( \frac{k}{\beta_2}\lam^{\beta_2}\mathbf{1}_{\{\beta_2=\beta_1\}} + \frac{1}{\beta_1}(1-\lam)^{\beta_1} \bigg)^{r_2}.
\end{aligned}$$

Therefore, we can derive that
$$\begin{aligned}
&\lim_{h\rightarrow+\infty} \frac{c^*(x_{\text{lavs}}(h),h)}{x_{\text{lavs}}(h)} = \lim\limits_{h\rightarrow+\infty} \frac{h}{x_\mathrm{lavs}(h)} \\
=& \lim\limits_{h\rightarrow+\infty} \frac{h}{-C_5(h)r_1(1-\lam)^{\beta_1(r_1-1)}h^{(\beta_1-1)(r_1-1)}
- C_6(h)r_2(1-\lam)^{\beta_1(r_2-1)}h^{(\beta_1-1)(r_2-1)} + \frac{h}{r}} \\
=& \bigg(1 - \frac{(1-\lam)^{\beta_1(r_1-1)}\gamma_1}{\gamma_1-r_2}
\bigg( \frac{k}{\beta_2}\lam^{\beta_2}\mathbf{1}_{\{\beta_2=\beta_1\}} + \frac{1}{\beta_1}(1-\lam)^{\beta_1} \bigg)^{r_2}\bigg)^{-1}r ,
\end{aligned}$$
and
$$\begin{aligned}
&\lim_{h\rightarrow+\infty}\frac{\pi^*(x_{\text{lavs}}(h),h)}{x_{\text{lavs}}(h)}
=\lim_{h\rightarrow+\infty}\frac{\pi^*(x_{\text{lavs}}(h),h)}{h}\cdot \frac{h}{x_{\text{lavs}}(h)}\\
=& \frac{2r}{\mu-r}\lim\limits_{h\rightarrow+\infty} \frac{h}{x_\mathrm{lavs}(h)}\times \lim\limits_{h\rightarrow+\infty} \frac{(1-\lam)^{\beta_1(r_1-1)}C_5(h)h^{-r_1-r_2\beta_1} + (1-\lam)^{\beta_1(r_2-1)}C_6(h)h^{-r_2-r_1\beta_1}}{h} \\
=& \frac{2r}{\mu-r} \times \bigg(1 - \frac{(1-\lam)^{\beta_1(r_1-1)}\gamma_1}{\gamma_1-r_2}
\bigg( \frac{k}{\beta_2}\lam^{\beta_2}\mathbf{1}_{\{\beta_2=\beta_1\}} + \frac{1}{\beta_1}(1-\lam)^{\beta_1} \bigg)^{r_2}\bigg)^{-1}  \\ &~~\times\frac{(1-\lam)^{\beta_1(r_1-1)}}{\gamma_1-r_2}\bigg( \frac{k}{\beta_2}\lam^{\beta_2}\mathbf{1}_{\{\beta_2=\beta_1\}} + \frac{1}{\beta_1}(1-\lam)^{\beta_1} \bigg)^{r_2}.
\end{aligned}$$

Let us then consider the other case when $y_1(h) > y_2(h)$. If $\beta_2<\beta_1$, the second term in \eqref{eq: tagent z_eq} converges to 0, and thus $\frac{w(h)}{h}$ converges to a constant $-\lam\gamma_1$.
If $\beta_2 = \beta_1$, the second term in \eqref{eq: tagent z_eq} equals a constant, and $\frac{w(h)}{h}$ becomes a constant $w(1)$ that is the unique solution to $-\frac1{\gamma_1}w(1)^{\beta_1}+\frac{k}{\beta_2}\lam^{\beta_2} - \lam w(1)^{\beta_1-1} = 0$.
Otherwise, if $\beta_2 > \beta_1$, the second term in \eqref{eq: tagent z_eq} goes to infinity as $h\rightarrow+\infty$, indicating that $\frac{w(h)}{h}$ converges to 0.

Thus, we always have that
$$\begin{aligned}
\lim\limits_{h\rightarrow+\infty} C_3(h)h^{-r_1-r_2\beta_1}&= \lim\limits_{h\rightarrow+\infty}\left[\frac{r_2}{r(r_1-r_2)(\gamma_1-r_1)}\bigg( \frac wh \bigg)^{r_1+r_2\beta_1}
+ \frac{\lam}{r(r_1-r_2)}\bigg( \frac wh \bigg)^{r_2(\beta_1-1)}\right] \\
&= \frac{w_0^{r_2(\beta_1-1)} (r_2w_0 + \lam(\gamma_1-r_1))} {r(r_1-r_2)(\gamma_1-r_1)},
\end{aligned}$$
where $w_0 := \lim\limits_{h\rightarrow+\infty} \frac{w(h)}{h}$. It holds that
$$\begin{aligned}
\lim\limits_{h\rightarrow+\infty} C_5(h)h^{-r_1-r_2\beta_1}
&= \frac{w_0^{r_2(\beta_1-1)} (r_2w_0 + \lam(\gamma_1-r_1)) + (\gamma_1-1)(1-\lam)^{r_2\beta_1+r_1}}{r(r_1-r_2)(\gamma_1-r_1)}.
\end{aligned}$$
Then we deduce that
$$\begin{aligned}
&\lim_{h\rightarrow+\infty} \frac{c^*(x_{\text{lavs}}(h),h)}{x_{\text{lavs}}(h)}
= \lim\limits_{h\rightarrow+\infty} \frac{1}{-(1-\lam)^{\beta_1(r_1-1)}\frac{\gamma_1(r_1-r_2)}{\gamma_1-r_2}C_5(h)h^{-r_2\beta_1-r_1} + \frac1r} \\
=& \bigg(1 - \frac{\gamma_1(1-\lam)^{-r_2\beta_1}}{(\gamma_1-r_1)(\gamma_1-r_2)} \bigg( w_0^{r_2(\beta_1-1)} (r_2w_0 + \lam(\gamma_1-r_1)) + (\gamma_1-1)(1-\lam)^{r_2\beta_1+r_1}\bigg) \bigg)^{-1}r
\end{aligned}$$
and
$$\begin{aligned}
&\lim_{h\rightarrow+\infty}\frac{\pi^*(x_{\text{lavs}}(h),h)}{x_{\text{lavs}}(h)}\\
=& \frac{2r}{\mu-r}\lim\limits_{h\rightarrow+\infty} \frac{h}{x_\mathrm{lavs}(h)}\cdot (1-\lam)^{\beta_1(r_1-1)}\frac{r_1-r_2}{\gamma_1-r_2}\cdot
\frac{w_0^{r_2(\beta_1-1)} (r_2w_0 + \lam(\gamma_1-r_1)) + (\gamma_1-1)(1-\lam)^{r_2\beta_1+r_1}}{r(r_1-r_2)(\gamma_1-r_1)} \\
=& \frac{2r(1-\lam)^{-r_2\beta_1}}{\mu-r} \cdot \bigg(1 - \frac{\gamma_1(1-\lam)^{-r_2\beta_1}}{(\gamma_1-r_1)(\gamma_1-r_2)} \bigg( w_0^{r_2(\beta_1-1)} (r_2w_0 + \lam(\gamma_1-r_1)) + (\gamma_1-1)(1-\lam)^{r_2\beta_1+r_1}\bigg) \bigg)^{-1} \\
&\times \frac{ w_0^{r_2(\beta_1-1)} (r_2w_0 + \lam(\gamma_1-r_1)) + (\gamma_1-1)(1-\lam)^{r_2\beta_1+r_1} }{(\gamma_1-r_1)(\gamma_1-r_2)},
\end{aligned}$$
where
$$
w_0 = \begin{cases}
-\lam\gamma_1 , & \mbox{if } \beta_2<\beta_1\leq 1-\lam, \\
w(1), & \mbox{if } \beta_2 = \beta_1, \\
0 , & \mbox{if } \beta_2>\beta_1.
\end{cases}
$$

Recall that $\frac{\pi^*(x)}{x} = \frac{\mu-r}{\sigma^2(1-\beta_1)}$ and $\frac{c^*(x)}{x}= \frac{(\gamma_1-r_1)(\gamma_1-r_2)}{r_1r_2}r$ in the Merton's problem. In our setting, as $\lam \rightarrow0$, it is obvious that $\beta_1 < 1-\lam$.
On the other hand, similar to the discussion of the limit of $\frac{w(h)}{h}$ as $h\rightarrow+\infty$, we have that $w(1)\rightarrow0$ as $\lam \rightarrow0$, and thus $w_0\rightarrow0$ as $\lam\rightarrow0$ in all three scenarios when $y_1(h)>y_2(h)$ as $h\rightarrow+\infty$.
Therefore, we deduce that
$$\begin{aligned}
&\lim\limits_{\lam\rightarrow0} \lim_{h\rightarrow+\infty} \frac{c^*(x_{\text{lavs}}(h),h)}{x_{\text{lavs}}(h)} \\
=& \lim\limits_{\lam\rightarrow0}
\bigg(1 - \frac{\gamma_1(1-\lam)^{-r_2\beta_1}}{(\gamma_1-r_1)(\gamma_1-r_2)} \bigg( w_0^{r_2(\beta_1-1)} (r_2w_0 + \lam(\gamma_1-r_1)) + (\gamma_1-1)(1-\lam)^{r_2\beta_1+r_1}\bigg) \bigg)^{-1}r \\
=& \frac{(\gamma_1-r_1)(\gamma_1-r_2)}{r_1r_2}r,
\end{aligned}$$
and
$$\begin{aligned}
&\lim\limits_{\lam\rightarrow0} \lim_{h\rightarrow+\infty}\frac{\pi^*(x_{\text{lavs}}(h),h)}{x_{\text{lavs}}(h)}\\
=&\lim\limits_{\lam\rightarrow0} \frac{2r(1-\lam)^{-r_2\beta_1}}{\mu-r} \times \bigg(1 - \frac{\gamma_1(1-\lam)^{-r_2\beta_1}}{(\gamma_1-r_1)(\gamma_1-r_2)} \bigg( w_0^{r_2(\beta_1-1)} (r_2w_0 + \lam(\gamma_1-r_1))\\
& + (\gamma_1-1)(1-\lam)^{r_2\beta_1+r_1}\bigg) \bigg)^{-1} \times \frac{ w_0^{r_2(\beta_1-1)} (r_2w_0 + \lam(\gamma_1-r_1)) + (\gamma_1-1)(1-\lam)^{r_2\beta_1+r_1} }{(\gamma_1-r_1)(\gamma_1-r_2)} \\
=& \frac{2r(\gamma_1-1)}{(\mu-r)r_1r_2}
= \frac{2r(\gamma_1-1)}{-\frac{2r}{\kappa^2}(\mu-r)}
= \frac{\mu-r}{\sigma^2(1-\beta_1)},
\end{aligned}$$
which complete the proof.
\end{proof}

\subsection{Proof of Corollary \ref{thm: long_run}}
\begin{proof}
Let us consider the auxiliary process $Y_t^*:= Y_t(y^*)$ and $H_t^*$ defined in Theorem \ref{thm: beta<1}.

(i) The long-run fraction of time that the agent stays in the region $\{x_\mathrm{aggr}(H^*_t)\leq X^*_t \leq x_\mathrm{lavs}(H^*_t)\}$ can be computed by
$$
\begin{aligned}
&\lim\limits_{T\rightarrow+\infty}\frac{1}{T}\E\bigg[\int_0^T \mathbf{1}_{\{x_\mathrm{aggr}(H_t^*) < X_t \leq x_\mathrm{lavs}(H_t^*)\}}dt\bigg] \\
=&\lim\limits_{T\rightarrow+\infty}\frac{1}{T}\E\bigg[\int_0^T\mathbf{1}_{\{ y_3(H_t^*) \leq Y_t(y^*) < y_2(H_t^*) \}}dt\bigg]\\
=& \lim\limits_{T\rightarrow+\infty}\frac{1}{T}\E\bigg[\int_0^T\mathbf{1}_{\{\inf\limits_{s\leq t}Y_s(y^*)\leq  Y_t(y^*)< \lim\limits_{h\rightarrow+\infty} \frac{y_2(h)}{y_3(h)}\inf\limits_{s\leq t}Y_s(y^*) \}}dt\bigg] \\
=& 1-\lim\limits_{h\rightarrow+\infty} \frac{y_3(h)}{y_2(h)},
\end{aligned}
$$
where the last equation holds by the same argument to prove Theorem 5.1 in \cite{GuasoniHubermanR2020MFF}. 

(ii) The long-run fraction of time that the agent stays in the region $\{0\leq X^*_t\leq x_\mathrm{zero}(H^*_t)\}$ can be computed by
$$
\begin{aligned}
&\lim\limits_{T\rightarrow+\infty}\frac{1}{T}\E\bigg[\int_0^T \mathbf{1}_{\{ X_t < x_\mathrm{zero}(H_t^*)\}}dt\bigg] \\
=&1 - \lim\limits_{T\rightarrow+\infty}\frac{1}{T}\E\bigg[\int_0^T \mathbf{1}_{\{x_\mathrm{zero}(H_t^*)\leq X_t \leq x_\mathrm{lavs}(H_t^*)\}}dt\bigg] \\
=&1 - \lim\limits_{T\rightarrow+\infty}\frac{1}{T}\E\bigg[\int_0^T\mathbf{1}_{\{ y_3(H_t^*) \leq Y_t(y^*)\leq y_1(H_t^*) \}}dt\bigg]. \\
=&1 - \lim\limits_{T\rightarrow+\infty}\frac{1}{T}\E\bigg[\int_0^T\mathbf{1}_{\{\inf\limits_{s\leq t}Y_s(y^*)\leq  Y_t(y^*)\leq \lim\limits_{h\rightarrow+\infty} \frac{y_1(h)}{y_3(h)}\cdot \inf\limits_{s\leq t}Y_s(y^*) \}}dt\bigg] \\
=&\lim\limits_{h\rightarrow+\infty} \frac{y_3(h)}{y_1(h)}.
\end{aligned}
$$

(iii) Let $\tv(y,h)$ be the solution to the following PDE:
$$
\begin{aligned}
\frac{\kappa^2}{2}y^2\tv_{yy}(y,h) - \frac{\kappa^2}{2}y\tv_y(y,h) = -1, ~~\mathrm{for}~~(y,h)\in\Omega, \\
\tv(y_1(h), h) = 0,~~ \tv_h(y_3(h), h) = 0,
\end{aligned}
$$
where $\Omega = \{(y,h)\in\mathbb{R}_+^2: y_3(h) \leq y \leq y_1(h)\}$. It holds that
$\tv(y,h) = \overline{C}_1(h)y^2 + \overline{C}_2(h) + \frac{\log y}{\kappa^2}$,
where $\overline{C}_1(h)$ and $\overline{C}_2(h)$ satisfy
$$
\begin{aligned}
&\overline{C}_1(h)y_1(h)^2 + \overline{C}_2(h) + \frac{\log y_1(h)}{\kappa^2} = 0,\\
&\overline{C}_1'(h)y_3(h)^{2} + \overline{C}_2'(h) = 0.
\end{aligned}
$$
Applying It$\hat{\mathrm{o}}$'s formula to $\tv(Y_t(y^*), H_t^*)$, and integrating from 0 to $\tau_\mathrm{zero}$, we have that
$$
\begin{aligned}
&\tv(Y_{\tau_\mathrm{zero}}(y^*), H_{\tau_\mathrm{zero}}^*) - \tv(y^*, H_0^*)\\
=& -\tau_\mathrm{zero} - \kappa\int_0^{\tau_\mathrm{zero}} Y_s(y^*)\tv_y(Y_s(y^*), H_s^*)dWs + \int_0^{\tau_\mathrm{zero}} \tv_h(Y_s(y^*), H_s^*)dH_s^*.
\end{aligned}
$$
Note that $\tv(Y_{\tau_\mathrm{zero}}(y^*), H_{\tau_\mathrm{zero}}^*) = 0$, the stochastic integral is square-integrable and thus a martingale with zero mean, and $H_t^*$ only increases when $\tv_h(Y_s(y^*), H_s^*)=0$, implying $\int_0^{\tau_\mathrm{zero}} \tv_h(Y_s(y^*), H_s^*)dH_s^* = 0$.
Together with the fact that $y^* = f(x,h)$, we can finally deduce that
$\E[\tau_\mathrm{zero}] = \tv(f(x,h), h) = \overline{C}_1(h)f(x,h)^2 + \overline{C}_2(h) + \frac{\log f(x,h)}{\kappa^2}$.

(iv) Before time $\tau_\mathrm{lavs}$, the historical consumption peak $H_t^* = h$ does not increase, and
$$\{Y_t(y^*) \leq y_3(h) \} = \bigg\{ -\kappa W_t - \frac{\kappa^2}{2}t \leq -\log\bigg( \frac{y^*h^{1-\beta_1}}{(1-\lam)^{\beta_1}} \bigg) \bigg\}.$$
Then, by equation (9.1) in \cite{RogersWilliams1994book}, let $b=\frac{1}{\kappa}\log\big(\frac{y^*h^{1-\beta_1}}{(1-\lam)^{\beta_1}}\big)$, $c = \frac{\kappa}{2}$, $\beta = \sqrt{c^2+2\nu}-c$, it follows that for any $\nu>0$:
$\E[e^{-\nu\tau_\mathrm{lavs}}] = e^{-b\beta}$.
Then, it holds that
$$
\E[\tau_\mathrm{lavs}] = -\frac{d\E[e^{{\cblue-}\nu\tau_\mathrm{lavs}}]}{d\nu}\bigg|_{\nu\downarrow0} = \frac{b}{c}= \frac{2}{\kappa^2}\log\bigg(\frac{y^*h^{1-\beta_1}}{(1-\lam)^{\beta_1}}\bigg) =  \frac{2}{\kappa^2}\log\bigg(\frac{f(x,h)h^{1-\beta_1}}{(1-\lam)^{\beta_1}}\bigg).
$$

\end{proof}

\section*{Acknowledgements}
{We thank two anonymous referees for their helpful comments on the presentation of this paper. X. Li is partially supported by the Hong Kong General Research Fund under grants 15215319, 15216720 and 15221621. X. Yu is supported by the Hong Kong Polytechnic University research grant under no. P0031417.}

\ \\


\begin{thebibliography}{10}
{\small
\bibitem{Abel1990AER}
A.~B. Abel.
\newblock Asset prices under habit formation and catching up with the joneses.
\newblock {\em The American Economic Review}, 80(2):38--42, 1990.

\bibitem{AngBay}
B. Angoshtari, E. Bayraktar and V. Young.
\newblock Optimal dividend distribution under drawdown and ratcheting constraints on dividend rates.
\newblock {\em SIAM Journal on Financial Mathematics}, 10(2):547-577, 2019.

\bibitem{Arun2020arXiv}
T.~Arun.
\newblock The {M}erton problem with a drawdown constraint on consumption.
\newblock {\em Preprint, arXiv:1210.5205}, 2012.


\bibitem{BerklKouwbgPost2004RES}
A.~B. Berkelaar, R.~Kouwenberg, and T.~Post.
\newblock Optimal portfolio choice under loss aversion.
\newblock {\em Review of Economics and Statistics}, 86(4):973--987, 2004.

\bibitem{BLY2021SICON}
L.~Bo, H.~Liao and X.~Yu.
\newblock Optimal tracking portfolio with a ratcheting capital benchmark.
\newblock {\em SIAM Journal on Control and Optimization}, 59(3):2346--2380,
  2021.


\bibitem{Constantinides1990JPE}
G.~M. Constantinides.
\newblock Habit formation: A resolution of the equity premium puzzle.
\newblock {\em Journal of Political Economy}, 98(3):519--543, 1990.

\bibitem{Curatola2017QREF}
G.~Curatola.
\newblock Optimal portfolio choice with loss aversion over consumption.
\newblock {\em The Quarterly Review of Economics and Finance}, 66:345--358,
  2017.

\bibitem{DengLiPY2020arXiv}
S. Deng, X. Li, H. Pham and X. Yu.
\newblock Optimal consumption with reference to past spending maximum.
\newblock {\em Finance and Stochastics}, 26:217-266, 2022.

\bibitem{DetempleKaratzas2003JET}
J.~Detemple and I.~Karatzas.
\newblock Non-addictive habits: Optimal consumption-portfolio policies.
\newblock {\em Journal of Economic Theory}, 113:265--285, 2003.

\bibitem{DetempleZapatero1992MF}
J.~Detemple and F.~Zapatero.
\newblock Optimal consumption-portfolio policies with habit formation.
\newblock {\em Mathematical Finance}, 2(4):251--274, 1992.

\bibitem{DongZheng2020EJOR}
Y.~Dong and H.~Zheng.
\newblock Optimal investment with S-shaped utility and trading and value at
  risk constraints: An application to defined contribution pension plan.
\newblock {\em European Journal of Operational Research}, 281:341--356, 2020.


\bibitem{Dyb1995}
P. H. Dybvig.
\newblock Dusenberry's ratcheting of consumption: Optimal dynamic consumption and investment given intolerance for any decline in standard of living.
\newblock {\em The Review of Economic Studies}, 62(2):287--313, 1995.


\bibitem{EnglezosKaratzas2009Sicon}
N.~Englezos and I.~Karatzas.
\newblock Utility maximization with habit formation: Dynamic programming and
  stochastic {PDE}s.
\newblock {\em SIAM Journal on Control and Optimization}, 48(2):481--520, 2009.

\bibitem{GuasoniHubermanR2020MFF}
P.~Guasoni, G.~Huberman and D.~Ren.
\newblock Shortfall aversion.
\newblock {\em Mathematical Finance}, 30(3):869-920, 2020.

\bibitem{HeStrub2019Preprint}
X.~He and M.~Strub.
\newblock How endogenization of the reference point affects loss aversion: a study of portfolio selection.
\newblock {\em Operations Research}, 70(6):3035-3053, 2022.

\bibitem{HeYang2019MF}
X.~He and L.~Yang.
\newblock Realization utility with adaptive reference points.
\newblock {\em Mathematical Finance}, 29(2):409--447, 2019.

\bibitem{HeZhou2011MS}
X.~He and X.Y. Zhou.
\newblock Portfolio choice under cumulative prospect theory: An analytical
  treatment.
\newblock {\em Management Science}, 57(2):315--331, 2011.

\bibitem{HeZhou2014QF}
X.~He and X.Y. Zhou.
\newblock Myopic loss aversion, reference point, and money illusion.
\newblock {\em Quantitative Finance}, 14(9):1541--1554, 2014.

\bibitem{KLSX}
I. Karatzas, J. P. Lehoczky, S. E. Shreve and  G. L. Xu.
\newblock Martingale and duality methods for utility maximization in an incomplete market.
\newblock {\em SIAM Journal on Control and Optimization}, 29:702-730, 1991.


\bibitem{LYZ23}
X. Li, X. Yu and Q. Zhang.
\newblock Optimal consumption and life insurance under shortfall aversion and a drawdown constraint.
\newblock {\em Insurance: Mathematics and Economics}, 108:25--45, 2023

\bibitem{JinZhou2008MF}
H.~Jin and X.Y. Zhou.
\newblock Behavioral portfolio selection in continuous time.
\newblock {\em Mathematical Finance}, 18(3):385--426, 2008.

\bibitem{KahnemanTversky2013Book}
D.~Kahneman and A.~Tversky.
\newblock Prospect theory: An analysis of decision under risk.
\newblock In {\em Handbook of the fundamentals of financial decision making:
  Part I}, pages 99--127. World Scientific, 2013.


\bibitem{Mehra}
R. Mehra and E. C. Prescott.
\newblock The equity premium: A puzzle.
\newblock {\em Journal of Monetary Economics}, 15(2):145-161, 1969.

\bibitem{Mert1969RES}
R.~C. Merton.
\newblock Lifetime portfolio selection under uncertainty: The continuous time case.
\newblock {\em The Review of Economics and Statistics}, 51(3):247--257, 1969.



\bibitem{Mert1971JET}
R.~C. Merton.
\newblock Optimal consumption and portfolio rules in a continuous-time model.
\newblock {\em Journal of Economic Theory}, 3:373--413, 1971.

\bibitem{Reichlin2013MFE}
C.~Reichlin.
\newblock Utility maximization with a given pricing measure when the utility is
  not necessarily concave.
\newblock {\em Mathematics and Financial Economics}, 7(4):531--556, 2013.

\bibitem{RogersWilliams1994book}
L. C. G. Rogers and D. Williams.
\newblock Diffusions, markov processes and martingales, Volume 1: Foundations.
\newblock {\em John Wiley \& Sons, Ltd., Chichester}, 7, 1994.

\bibitem{SchroderSkiadas2002RFS}
M.~Schroder and C.~Skiadas.
\newblock An isomorphism between asset pricing models with and without linear
  habit formation.
\newblock {\em The Review of Financial Studies}, 15(4):1189--1221, 2002.

\bibitem{BilsenLaevenN2020MS}
S.~van Bilsen, R.~Laeven and E.~Nijman.
\newblock Consumption and portfolio choice under loss aversion and endogenous
  updating of the reference level.
\newblock {\em Management Science}, 66(9):3799--4358, 2020.

\bibitem{Yu2022}
Y. Yang and X.~Yu.
\newblock Optimal entry and consumption under habit formation.
\newblock {\em Advances in Applied Probability}, 54(2):433--459, 2022.

\bibitem{Yu2015AoAP}
X.~Yu.
\newblock Utility maximization with addictive consumption habit formation in
  incomplete semimartingale markets.
\newblock {\em Annals of Applied Probability}, 25(3):1383--1419, 2015.

\bibitem{Yu2017AoAP}
X.~Yu.
\newblock Optimal consumption under habit formation in markets with transaction
  costs and random endowments.
\newblock {\em Annals of Applied Probability}, 27(2):960--1002, 2017.
}
\end{thebibliography}
\end{document}